\documentclass[11pt]{amsart}

\usepackage[margin=2.5cm]{geometry}
\usepackage{color}
\usepackage{mathrsfs}
\usepackage{mathtools}
\usepackage{amsmath}
\usepackage{amssymb}
\usepackage{enumitem}
\usepackage{bbm}
\usepackage{esint}
\usepackage{nicefrac}
\numberwithin{equation}{section}
\usepackage[colorlinks,citecolor=green,linkcolor=red,urlcolor=black]{hyperref}

\usepackage[latin1]{inputenc}
\usepackage{tcolorbox}

\usepackage{todonotes}

\newtheorem{theorem}{Theorem}[section]
\newtheorem{corollary}[theorem]{Corollary}
\newtheorem{lemma}[theorem]{Lemma}

\theoremstyle{definition}
\newtheorem{definition}[theorem]{Definition}

\newtheorem{remark}[theorem]{Remark}
\newtheorem{question}[theorem]{Question}

\newcommand{\N}{\mathbb{N}}

\newcommand{\R}{\mathbb{R}}
\renewcommand{\S}{\mathbb{S}}
\newcommand{\Z}{{\rm Z}}

\newcommand{\sfd}{{\sf d}}

\newcommand{\rr}{\mathbb R}
\newcommand{\restr}[1]{\lower3pt\hbox{$|_{#1}$}}

\newcommand{\eps}{\varepsilon}  
\newcommand{\nchi}{{\raise.3ex\hbox{$\chi$}}}
\newcommand{\weakto}{\rightharpoonup}

\newcommand{\fr}{\hfill$\blacksquare$}  

\newcommand{\LIP}{\mathrm{LIP}}
\newcommand{\Lip}{\mathrm{Lip}}
\newcommand{\lip}{\mathrm{lip}}
\newcommand{\esssup}{{\rm ess}\sup}
\newcommand{\essinf}{{\rm ess}\inf}
\newcommand{\diam}{\mathrm{diam}}

\newcommand{\RCD}{{\sf RCD}}

\newcommand{\mm}{\mathfrak m}

\renewcommand{\limsup}{\varlimsup}
\renewcommand{\liminf}{\varliminf}
\renewcommand{\d}{{\rm d}}
\newcommand{\X}{{\rm X}}

\newcommand{\Xdm}{(\X,\sfd,\mm)}
\newcommand{\rmCh}{{\rm Ch}}

\newcommand{\supp}{{\rm supp}}

\newcommand{\PP}{\mathscr{P}}

\renewcommand{\phi}{\varphi}
\newcommand{\avr}{{\sf AVR}}

\newcommand{\mres}{\mathbin{\vrule height 1.6ex depth 0pt width 0.13ex\vrule height 0.13ex depth 0pt width 1.3ex}}

\title[]{An overview of the stability of Sobolev inequalities on Riemannian manifolds with Ricci lower bounds}

\author[]{Francesco Nobili} 
\address{Universit\'a di Pisa, Dipartimento di Matematica, Largo Bruno Pontecorvo 5,
56127 Pisa, Italy}
\email{\url{francesco.nobili@dm.unipi.it}}

\begin{document}

\begin{abstract}
We review recent results regarding the problem of the stability of Sobolev inequalities on Riemannian manifolds with Ricci curvature lower bounds. We shall describe techniques and methods from smooth and non-smooth geometry, the fruitful combination of which revealed particularly effective. Furthermore, we present a self-contained overview of the proof of the stability of the Sobolev inequality on manifolds with non-negative Ricci curvature and Euclidean volume growth, adopting a direct strategy tailored to this setting. Finally, we discuss related stability results and present some open questions.
\end{abstract}
\maketitle
\allowdisplaybreaks
\setcounter{tocdepth}{2}
\tableofcontents

\section{Introduction}
In the recent series of works \cite{NobiliViolo21,NobiliViolo24,NobiliParise24,NobiliViolo24_PZ}, several results around functional inequalities in the context of Riemannian manifolds have been investigated with the primary focus on the problem of the stability of Sobolev inequalities. In this note, we report on the main results achieved and give a {self-contained overview of the techniques and methods. Emphasis will be given to Riemannian manifolds with Ricci curvature lower bounds. }

We start by introducing this fascinating problem while gradually moving from the classical Euclidean setting to that of Riemannian manifolds. In the meantime, we recall previous key contributions around the topic and then present our main results.
\subsection{The stability program on the Euclidean space}
The Sobolev inequality in sharp form on the Euclidean space $\R^d$, for $d \in \N$ {and $p \in (1,d)$}, reads
\begin{equation}
\|u\|_{L^{p^*}(\R^d)} \le S_{d,p}\| \nabla u\|_{L^p(\R^d)}, \qquad \forall u \in \dot W^{1,p}(\R^d),
\label{eq:SobEuclidea}
\end{equation}
where {$p^*\coloneqq pd /(d-p)$ is the critical exponent}, $\dot W^{1,p}(\R^d) \coloneqq \{ u \in L^{p^*}(\R^d) \colon |\nabla u|\in L^p(\R^d)\}$ and 
\begin{equation}
  S_{d,p} \coloneqq \frac 1d\left(\frac{d(p-1)}{d-p}\right)^{\frac{p-1}p}\left(\frac{\Gamma(d+1)}{d\omega_d\Gamma(d/p)\Gamma (d+1-d/p)}\right)^{\frac 1d}, \label{eq:Sobolev constant}  
\end{equation}
is the sharp Sobolev constant explicitly computed by Aubin \cite{Aubin76-2} and Talenti \cite{Talenti76}, where $\omega_d \coloneqq \pi^{d/2} / \Gamma(d/2+1)$ and $\Gamma(\cdot)$ is the Gamma function. Extremal functions, i.e.\ functions for which equality occurs in \eqref{eq:SobEuclidea}, are also completely characterized 
\begin{equation}
{u_{a,b,z_0}} \coloneqq \frac{a}{(1+b|\cdot -z_0|^\frac{p}{p-1})^{\frac{d-p}{p}}}, \qquad a \in \R,\, b>0,\, z_0 \in \R^d,
\label{eq:bubble intro}
\end{equation}
and are usually called \emph{Euclidean bubbles} (or Aubin-Talenti bubbles) because of their radial shape.

A natural question, firstly asked in \cite{BrezisLieb83}, is the following:
\begin{center}
     \em If $u$ almost saturates \eqref{eq:SobEuclidea}, is $u$ close to the family of Euclidean bubbles?
\end{center}
One of the main difficulties behind this question is that the extremal family \eqref{eq:bubble intro} is invariant for a non-compact group of symmetries. In fact, a qualitative answer to this question can be given by means of the celebrated concentration compactness of Lions \cite{Lions84,Lions85}. If $(u_n)\subseteq \dot W^{1,p}(\R^d)$ non zero is so that $S_{d,p} - \|u_n\|_{L^{p^*}(\R^d)} /\| \nabla u\|_{L^p(\R^d)} \to 0$, then up to translations, dilations and scalings, it converges strongly in $L^{p^*}$ to some Euclidean bubble. In particular, by rephrasing this principle we get for $\eps>0$ the existence of some $\delta\coloneqq \delta(\eps,d,p)>0$ \emph{non-quantified} so that for every $0\neq u \in \dot W^{1,p}(\R^d)$ it holds
\[
\frac{\|u\|_{L^{p^*}(\R^d)}} {\| \nabla u\|_{L^p(\R^d)}} > S_{d,p}- \delta \qquad \implies \qquad \inf_{a,b,z_0}\frac{\| \nabla(u- u_{a,b,z_0})\|_{L^p(\R^d)} }{\|\nabla u\|_{L^{p}(\R^d)}} \le \eps,
\]
This provides a \emph{qualitative} answer to the stability problem.

\medskip

A \emph{quantitative} stability result instead aims at quantifying how close the function $u$ is to the family of Euclidean bubbles in terms of the deficit in the Sobolev inequality. The first result in this direction was obtained by Bianchi-Egnell \cite{BianchiEgnell91} who showed
\begin{equation}
    \inf_{a,b,z_0} \frac{\| \nabla( u -u_{a,b,z_0} )\|_{L^2(\R^d)}}{\|\nabla u\|_{L^2(\R^d)}} \le  C_d \Big(\frac{\| \nabla u\|_{L^2(\R^d)} }{ \|u\|_{L^{2^*}(\R^d)} }-S_{d,2}^{-1} \Big)^{\frac 12}, \qquad \forall u \in \dot W^{1,2}(\R^d), \label{eq:strong stability Eucl}
\end{equation}
for an implicit dimensional constant $C_d>0$. Subsequently, this research program was pushed further to deal with the much harder exponent $p \in (1,d)$ in the series of works  \cite{CianchiFuscoMaggiPratelli09,FigalliNeumayer19,Neumayer19,FigalliZhang22} requiring a new set of ideas. More recently, there has been a breakthrough in quantifying explicitly the stability constant $C_d$ appearing in \eqref{eq:strong stability Eucl}, see \cite{DolbeaultEstebanFigalliFrankLoss23,DolbeaultEstebanFigalliFrankLoss24} (also the related \cite{Konig23} and  \cite{Frank23_Survey}). 

Finally, we mention that {for $p=2$ the stability for signed \emph{critical points} of \eqref{eq:SobEuclidea} has also been studied. The family of signed critical points, i.e.\ nonnegative solutions of
\[
S_{d,2}^2\Delta u +u|u|^{2^*-2}=0,
\]
are again classified and turn out to be formed by Euclidean bubbles \cite{GidasNiNiremberg79}. Starting from \cite{Struwe84}, qualitative stability principles were established roughly saying that almost nonnegative solutions to the above partial differential equation are close to a finite sum of Euclidean bubbles. The latter phenomenon is usually referred to as bubbling. Besides, quantitative improvements of this principle were then obtained in \cite{CiraoloFigalliMaggi18,FigalliGlaudo20,DengSunWei21}.}

\subsection{The stability program on Riemannian manifolds: some examples}
Exporting this research line to the setting of Riemannian manifolds presents some interesting challenges that we shall face in the following order:
\begin{itemize}
    \item[{\rm i)}] understand the validity and the sharpness of Sobolev inequalities in connection with geometric and curvature assumptions;
    \item[{\rm ii)}] study the existence of extremal functions;
    \item[{\rm iii)}] depending on existence/non-existence, investigate a stability principle.
\end{itemize}
As for i), starting from Aubin \cite{Aubin98} with the celebrated AB-program, there were big efforts in understanding Sobolev inequalities on closed Riemannian manifolds (i.e.\ compact and boundaryless). { On a closed Riemannian manifold $(M,g)$, the typical critical Sobolev inequality (for $p=2$) reads
\[
\|u\|_{L^{2^*}(M)}^2 \le A\|\nabla u\|_{L^2(M)}^2 + B \| u\|_{L^2(M)}^2,\qquad \forall u \in W^{1,2}(M),
\]
for  two constants $A,B>0$. Notice that the $B$-term is clearly necessary due to the presence of constant Sobolev functions. The first main goals of the AB-program are to discuss notions of optimal Sobolev constants and to estimate those in terms of geometric and curvature assumptions. Due to the presence of two parameters, there are two parallel ways of defining Sobolev optimal constants. The A-part of this program looks for the best possible $A$ for an inequality of this kind to hold, while the $B$-part of this program asks the opposite question. A successive important goal of the AB-program consists then in computing their values, opening the way to the study of the existence of extremal functions. We refer to the authoritative book \cite{Hebey99}, as well as to \cite{DH02}. A small introduction will be also given in Section \ref{sec:AB-program}.}

Instead, on non-compact manifolds, the validity of Sobolev embeddings is usually related to non-collapsing assumptions of volumes and Ricci curvature lower bounds, see \cite[Chapter 3]{Hebey99} for more insights and references. Starting from \cite{Ledoux00,Xia01}, it became clear that reasonable assumptions are the requirements of non-negative Ricci curvature and Euclidean volume growth, as we will see later.

As for ii)-iii), results in both directions are very rare. Looking at model spaces with non-zero constant sectional curvature, we recall that on the round sphere $\S^d$ it can be shown via stereographic projection (see, e.g., \cite[Appendix A]{FrankLieb12}) that \eqref{eq:SobEuclidea} as well as the class \eqref{eq:bubble intro} are in one-to-one correspondence with the following sharp Sobolev inequality studied in \cite{Aubin76-3}
\begin{equation}
\|u\|_{L^{2^*}(\nu)}^2\le \frac{2^*-2}{d}\| \nabla u\|_{L^{2}(\nu)}^2 + \| u\|_{L^2(\nu)}^2, \qquad \forall u \in W^{1,2}(\S^d),
\label{eq:Sob sphere}
\end{equation}
admitting the class of extremizers (see \cite[Chapter 5]{Hebey99})
\begin{equation}
 v_{a,b,z_0} := \frac{a}{(1-b\cos(\sfd(\cdot ,z_0))^{\frac{2-d}{2}}},\qquad \text{with } a\in\R,\, b\in[0,1),\, z_0 \in \S^d.
\label{eq:spherical bubbles}
\end{equation}
We shall refer to any of the above functions as a \emph{spherical bubble}. Here $\sfd$ is the geodesic distance on $\S^d$ with respect to the standard round metric and $\nu \coloneqq {\rm Vol}/{\rm Vol}(\S^d)$ is the renormalized volume form. This renormalization is not always standard, but we shall adopt it to get clearer inequalities. Notice that in this case, it holds
\[
 S_{d,2} = \left(\frac{2^*-2}{d}\right)^{\frac 12} {\rm Vol}(\S^d)^{-\frac{1}{d}}.
\]
In particular, the Bianchi-Egnell stability \eqref{eq:strong stability Eucl} carries over to this setting (again, by stereographic projection, see \cite[Eq. (3)]{EngelsteinNeumayerSpolaor22}) hence, the program of stability for \eqref{eq:Sob sphere} on the round sphere is completely understood.

An analogous program was then run in the Hyperbolic space for the sharp Poincar\'e-Sobolev inequality studied in \cite{ManciniSandeep08}. Stability results were recently derived in \cite{BhaktaGangulyKarmakarMazumdar22} (see also \cite{BhaktaGangulyKarmakarMazumdar23PartI} for critical points).

{
A much more general quantitative result has been achieved in \cite{EngelsteinNeumayerSpolaor22} for a class of \emph{conformally invariant Sobolev inequalities} related to the Yamabe problem \cite{Yamabe60} (see \cite{LP87,BrendleMarques11} and references therein). See also the recent quantitative analysis performed in \cite{ChenKim24}. We shall give more insights about these results in Section \ref{sec:AB-program}.}
\subsection{Main results}
In this part, we start specializing our discussion to the setting of Riemannian manifolds with Ricci curvature lower bounds.

\medskip

Let us first consider a $d$-dimensional Riemannian manifolds $(M,g)$, $d \ge 2$, satisfying
\begin{equation}
    {\sf Ric}_g \ge 0, \qquad {\sf AVR }(M)\coloneqq \lim_{R\to\infty}\frac{{\rm Vol}_g(B_{R}(x))}{\omega_d R^d }>0, \label{eq:AVR}
\end{equation}
for $x \in M$. The second condition is called \emph{Euclidean volume growth} property while ${\sf AVR }(M)$ is called the asymptotic volume ratio of $M$. Thanks to the Bishop-Gromov monotonicity, it holds that ${\sf AVR}(M)\in[0,1]$ and the limit exists and it is independent of $x$. We recall that ${\sf AVR}(M)=1$ occur if and only if $(M,g)$ is isometric to $\R^d$, see \cite{Colding97,DePhilippisGigli18}. 

{ The class of manifolds satisfying \eqref{eq:AVR} is rich and contains many examples besides $\R^d$ with the Euclidean metric. Notable examples with ${\sf AVR}(M) \in (0,1)$ are $d$-dimensional Ricci flat asymptotical locally Euclidean manifolds of dimension $d\ge 4$. In the case $d=4$, they arise as gravitational instantons in Euclidean quantum gravity theory \cite{Hawking77} and a concrete example is given by the Eguchi-Hanson metric \cite{EguchiHanson79}. Further examples are given in \cite{Menguy00} with infinite topological type. Besides, spaces satisfying \eqref{eq:AVR} constitute an important class in geometric analysis and, as we will see later, further examples are also weighted convex cones (see \cite{CabreRosOtonSerra16} and references therein) and cones arising as limits of manifolds (see Section \ref{sec:3 ingredients}).

For our goals, it is well known that coupling a Ricci lower bound with volume noncollapsing of geodesic balls is sufficient to guarantees the validity of Sobolev embeddings on noncompact manifolds (see \cite[Chapter 3]{Hebey99} for a detailed discussion). However, the stricter requirements contained in \eqref{eq:AVR} leads to \emph{sharp} Sobolev inequalities paving the way to a further stability analysis. Indeed,} in \cite{BaloghKristaly21} (see also \cite{Kristaly23} for an Optimal Transport proof revisiting \cite{C-ENV04}) the following sharp Euclidean Sobolev inequality was derived under the assumptions \eqref{eq:AVR} and for all $p\in(1,d)$:
\begin{equation}
\| u\|_{L^{p^*}(M)} \le {\sf AVR }(M)^{-\frac{1}{d}}S_{d,p}\| \nabla u\|_{L^{p}(M)},\qquad \forall u \in \dot W^{1,p}(M).
\label{eq:AVR sob intro}
\end{equation}
{We recall that on manifold with ${\sf Ric}_g\ge 0$, the Euclidean volume growth property turns out to be also a necessary assumption to study Sobolev inequalities, as understood since the works \cite{Ledoux99,Xia01}}. 

\medskip

The natural follow-up question is whether extremal functions for \eqref{eq:AVR sob intro} exist or not. This has been fully addressed in \cite{NobiliViolo24_PZ}.
\begin{theorem}\label{thm:rigidity sob main avr}
    Let $(M,g)$ be a $d$-dimensional non-compact Riemannian manifolds, $d\ge 2$, with ${\sf Ric}_g\ge 0$ and ${\sf AVR}(M)>0$. Fix $p \in (1,d)$. Then, equality holds in \eqref{eq:AVR sob intro} for some non-zero $u \in \dot W^{1,p}(M)$ if and only if $M$ is isometric to $\R^d$. 
\end{theorem}
A straightforward corollary is then the following.
\begin{corollary}\label{cor:non existence}
    Let $(M,g)$ be a $d$-dimensional non-compact Riemannian manifolds, $d\ge 2$, with ${\sf Ric}_g\ge 0$. Fix $p \in (1,d)$. If ${\sf AVR}(M) \in (0,1)$, then there are no nonzero extremal functions in \eqref{eq:AVR sob intro}.
\end{corollary}
The above were obtained for $p=2$ in \cite{CatinoMonticelli22} by studying the rigidity of signed critical points for Sobolev inequalities and in  \cite{NobiliViolo24} by studying the rigidity of the P\'olya-Szeg\H{o} inequality. In \cite{NobiliViolo24_PZ}, we generalized this result for all $p \in (1,d)$ by a fine analysis of the rigidity of the P\'olya-Szeg\H{o} inequality in this setting. We will comment on this improvement later in Section \ref{sec:Polya}.

\medskip 
The next follow-up question is what happens when a function almost saturates \eqref{eq:AVR sob intro}. Notice that the question still makes sense, even though extremal functions might not exist, as \eqref{eq:AVR sob intro} is sharp on every $(M,g)$ as in \eqref{eq:AVR}, possibly non-isometric to $\R^d$. Our main result was obtained in \cite{NobiliViolo24}.
\begin{theorem}\label{thm:qualitative SobAVR intro} 
For every $\eps>0,V\in(0,1)$ and $d> 2,$ there exists $\delta\coloneqq \delta(\eps,d,V)>0$ such that the following holds.
Let $(M,g)$ be a $d$-dimensional non-compact Riemannian manifold with ${\sf Ric}_g\ge 0$ and ${\sf AVR}(M)>V$ and suppose there exists $0\neq u\in \dot W^{1,2}(M)$ satisfying
\[
  \frac{\| u\|_{L^{2^*}(M)}}{\| \nabla u\|_{L^2(M)}}> {\sf AVR}(M)^{-\frac 1d}S_{d,2}- \delta.
\]
Then, there exist $a \in \R,\, b>0,$ and $z_0 \in M$ so that
\[
    \frac{\| \nabla( u -u_{a,b,z_0} )\|_{L^2(M)}}{\|\nabla u\|_{L^2(M)}} \le  \eps,
\] 
where $u_{a,b,z_0} \coloneqq a(1+b\sfd_g(\cdot,z_0)^2)^{\frac{2-d}{2}}$.
\end{theorem}
We point out that in general, the family $u_{a,b,z_0}$is not a family of extremizers (recall Corollary \ref{cor:non existence}). Anyhow, we shall refer to any of the above as Euclidean bubbles. Finally, notice that Theorem \ref{thm:rigidity sob main avr} is stated only for $p=2$, as it was obtained in \cite{NobiliViolo24} only for this exponent (see Remark \ref{rmk:regularity PDE}).

\medskip 

We pass to a second setting of spaces with positive Ricci curvature lower bounds. Let us consider a closed $d$-dimensional Riemannian manifold $(M,g)$, $d>2,$ satisfying
\[
{\sf Ric}_g \ge (d-1).
\]
Recall that, by the maximal diameter theorem it holds ${\rm diam}(M)\le \pi$. Under these assumptions, the same Sobolev inequality \eqref{eq:Sob sphere} as in the sphere  is valid:
\begin{equation}
\|u\|_{L^{2^*}(\nu)}^2\le \frac{2^*-2}{d} \| \nabla u\|_{L^{2}(\nu)}^2 + \| u\|_{L^2(\nu)}^2, \quad \forall u \in W^{1,2}(M),
\label{eq:Sob main intro}
\end{equation}
where all the norms are computed with the renormalized volume form $\nu = {\rm Vol}_g/{\rm Vol}_g(M)$.  This result dates back to \cite{Said83} but several different proofs are by now available \cite{Bakry94,BL96,Fontenas97,Hebey99,BakryGentilLedoux14,DGZ20}. 
Since constant functions are Sobolev, as $M$ is compact, it is not possible to replace any constant $c<1$ in front of the $L^2$ term, otherwise the inequality would be violated.
Hence, we can introduce a notion of optimal Sobolev constant by setting
\begin{equation}
A^{\rm opt}(M) \coloneqq \sup_{\substack{u \in W^{1,2}(M), \\ u\neq cst.}} \frac{\|u\|_{L^{2^*}(\nu)}^2- \| u\|_{L^2(\nu)}^2}{\| \nabla u\|_{L^{2}(\nu)}^2 }.
\label{aopt intro}
\end{equation}
We refer again to \cite{Hebey99} and Section \ref{sec:AB-program} for more details. Therefore, now \eqref{eq:Sob main intro} takes simply the form
\begin{equation}
    A^{\rm opt}(M)\le \frac{2^*-2}{d} \overset{\eqref{eq:Sob sphere}}{=} A^{\rm opt}(\S^d). \label{intro: AM < AS}
\end{equation}
At this point, the program of stability for \eqref{eq:Sob main intro} is at a fork. Either one investigates the existence and stability properties of extremal functions $ \mathcal M(A)$ for $A^{\rm opt}(M)$, i.e.\ $v \in \mathcal M(A)$ if $v$ is a maximizers of \eqref{aopt intro}, or one instead tries to understand the rigid scenario in the comparison inequality \eqref{intro: AM < AS} and proceeds towards a slightly different stability program. The former approach, which has been undertaken in \cite{NobiliParise24}, produced a quantitative stability estimate of the following kind
\[
        \frac{ A^{\rm opt}(M)\|\nabla u\|^2_{L^2(\nu)} +\|u\|_{L^2(\nu)}^2}{\| u\|_{L^{2^*}(\nu)}^2} - 1 \ge C \left( \inf_{v \in {\mathcal M}(A)} \frac{\| u- v\|_{W^{1,2}(M)} }{\|u\|_{W^{1,2}(M)}} \right) ^{2+\gamma}
\]
for some constants $C>0,\gamma\ge 0$ depending on the manifold and under a suitable strict bound on the optimal constant $A^{\rm opt}(M)$. See Theorem \ref{thm:main quant compact spaces} for the precise (more general) statement. 

Instead, given the emphasis on Ricci lower bounds in this note, we next explore the latter approach. The combination of \cite{NobiliViolo21,NobiliViolo24,NobiliViolo24_PZ} gives:
\begin{theorem}\label{thm:into main compact}
    For every $d>2$, $\eps>0$, there is $\delta\coloneqq \delta(\eps,d)>0$ and a constant $C_d>0$ such that: for every $d$-dimensional Riemannian manifold $(M,g)$ with ${\sf Ric}_g\ge (d-1)$, it holds:
    \begin{itemize}
        \item[{\rm i)}] {\sc Rigidity}.   $A^{\rm opt}(M) = A(\S^d)$ if and only if $M$ is isometric  to the round sphere $\S^d$; 
        \item[{\rm ii)}] {\sc Quantitative geometric stability}. It holds
        \[
        C_d(\pi - \diam(M))^d\le  A(\S^d)-A^{\rm opt}(M);
        \]
        \item[{\rm iii)}]  {\sc Qualitative functional stability}. If for some non-constant $u \in W^{1,2}(M)$ it holds
        \[
           \frac{\|u\|_{L^{2^*}(\nu)}^2- \| u\|_{L^2(\nu)}^2}{\| \nabla u\|_{L^{2}(\nu)}^2 }  >  \frac{2^*-2}{d}-\delta,
        \]
        then there exist $a \in \R$, $b \in [0,1)$ and $z_0 \in M$ such that
        \begin{equation}\label{eq:close to bubble}
           \frac{\| \nabla(u- v_{a,b,z_0} )\|_{L^2(\nu)} + \|u-u_{a,b,z_0}\|_{L^{2^*}(\nu)}}{\|u\|_{L^{2^*}(\nu)}} \le \eps,
       \end{equation}
    where $v_{a,b,z_0}\coloneqq a(1-b\cos(\sfd_g(\cdot ,z_0))^{\frac{2-d}{2}}$.
    \end{itemize}
\end{theorem}
Again, we point out the crucial fact that now the family of functions $v_{a,b,z}$ appearing in \eqref{eq:close to bubble} is explicit but, in general, they are not maximizers for $A^{\rm opt}(M)$ in \eqref{aopt intro}. As they share the same profiles of the extremal family on $\S^d$ (recall \eqref{eq:Sob sphere},\eqref{eq:spherical bubbles}), we shall refer to them as spherical bubbles. 

\subsection{Overview of the methods and organization}
In this part, we give a quick overview of the main strategy to prove the main functional stability results, and then describe how this note is organized.  We shall mainly restrict our attention to the proof of Theorem \ref{thm:qualitative SobAVR intro}. This result is easier to obtain compared to Theorem \ref{thm:into main compact} but contains all the essential ideas.

\medskip 

Arguing by contradiction in Theorem \ref{thm:qualitative SobAVR intro}, { for every $\eps>0$} we get a sequence of manifolds satisfying for all $n\in\N$ 
\[
    (M_n,g_n) \text{ with } {\sf Ric}_{g_n} \ge 0,\qquad {\sf AVR}(M_n) >V,
\]
and functions $0\neq u_n \in \dot W^{1,2}(M_n)$, without loss of generality with $\|u_n\|_{L^{2^*}(M_n)} =1$, so that {
\[
\frac{\|u_n\|_{L^{2^*}(M_n)}}{\|\nabla u_n\|_{L^2(M_n)}} - {\sf AVR}(M_n)^{-\frac 1d}S_{d,2}\to 0,\qquad \text{as }n\uparrow\infty,
\]
but satisfying for all $n\in\N$ the following
\[
    \inf_{a,b,z_0} \|\nabla (u_n-u_{a,b,z_0})\|_{L^2(M_n)}>\eps.
\]
}As for the standard Lions concentration compactness strategy \cite{Lions84,Lions85}, one needs to look at the sequence of probability measures $\nu_n \coloneqq |u_n|^{2^*}{\rm Vol}_{g_n}$ and study its behavior. The major issues are: $(u_n)$ might fail to be pre-compact in $L^{2^*}$, exactly because of the symmetries that are acting on the family of Euclidean bubbles; the manifolds $M_n$ are not fixed, suggesting that one needs to allow suitable non-smooth limits and employ a generalized Sobolev calculus.

In a nutshell, the idea will be to work with a generalized notion of manifolds with Ricci curvature lower bounds to gain stability and pre-compactness results both at the level of the sequence of manifolds and the sequence of functions. In particular, our main proof argument for the stability results will consist of the combination of three main ingredients:
\begin{itemize}
    \item[{\rm i)}] Calculus and non-smooth spaces with Ricci curvature lower bounds in a synthetic sense;
    \item[{\rm ii)}] Fine P\'olya-Szeg\H{o} rearrangement inequalities with applications to sharp Sobolev inequalities on non-smooth spaces;
    \item[{\rm iii)}] Generalized concentration compactness principles on varying spaces.
\end{itemize}
To each of the above, we shall dedicate a section to give an overview of the main methods and further interesting details. With respect to the original works where these results were obtained, we shall eventually simplify some statements and give proofs or overviews of proofs adopting, when possible, more direct approaches.  

\medskip

This note is organized as follows: in Section \ref{sec:3 ingredients} we expand on the three ingredients discussed above; in Section \ref{sec:sketch}
we present our main proof-argument to deduce Theorem \ref{thm:qualitative SobAVR intro} where all the three ingredients are combined, and subsequently comment on the modifications of this strategy to obtain the functional stability result iii) in Theorem \ref{thm:into main compact}; finally, in Section \ref{sec:further}, we present related quantitative stability results and conclude by presenting some open questions.
\section{Three main ingredients}\label{sec:3 ingredients}
The aim of this section is to present the three main ingredients that combined yield the proof of Theorem \ref{thm:qualitative SobAVR intro}. We shall also expand on some relevant details.
\subsection{Calculus and {\sf RCD}-spaces}
In this part, we shall introduce the notion of Sobolev functions in the category of metric measure spaces and then consider suitable \emph{synthetic} generalizations of Riemannian manifolds with Ricci curvature lower bounds. We also give a brief overview of this research field focusing only on the relevant properties for our stability program. 

\medskip

Let us start by discussing the first-order calculus on metric measure spaces and we set up some notation. We recall that a metric measure space is a triple  $(\X,\sfd,\mm)$ where $(\X,\sfd)$ is a complete and separable metric space and $\mm$ is a non-negative, non-zero and boundedly finite Borel measure. We denote $C(\X),C_b(\X),C_{bs}(\X)$ respectively the space of continuous functions, continuous and bounded functions and continuous and boundedly supported functions on $\X$. By $\Lip(\X),\Lip_{bs}(\X)$, we denote respectively the collection of Lipschitz functions and boundedly supported Lipschitz functions and by $\lip(u)$ the local Lipschitz constant of $u\colon \X\to \R$. By $\PP(\X)$ we denote the collection of Borel probability measures on $\X$. For any $p \in (1,\infty)$, we denote by $L^p(\X),L^p_{loc}(\X)$ respectively the space of $p$-integrable functions and $p$-integrable functions on a neighborhood of every point (up to $\mm$-a.e.\ equality relation) on $\X$. Equivalently, we shall also write $L^p(\mm),L^p_{loc}(\mm)$. 

Next, we define the \emph{Cheeger} energy
\[
\rmCh_p(u)\coloneqq \inf\left\{ \int \lip(u_n)^p\,\d \mm \colon (u_n)\subseteq \Lip(\X), \, u_n \to u \text{ in }L^p(\mm)\right\},
\]
setting the infimum to $+\infty$ as customary when no such sequence $(u_n)$ exists. Then, the Sobolev space $W^{1,p}(\X)$ is defined as the collection of $u \in L^p(\mm)$ so that $\rmCh_p(u)<\infty$ equipped with the usual norm $\|u\|_{W^{1,p(\X)}}^p \coloneqq \|u\|_{L^p(\mm)}^p + \rmCh_p(u)$. Relevant properties of the Sobolev calculus have been deeply analyzed in the literature (see \cite{Cheeger00,Shan00}), here we are following the equivalent axiomatization given by \cite{AmbrosioGigliSavare11-3}. Accordingly, we recall that the Cheeger energy, when finite, admits the representation
\[
\rmCh_p(u) = \int |\nabla u|^p\,\d \mm,
\]
for a suitable function $|\nabla u| \in L^p(\X)$ called minimal $p$-weak upper gradient. { In particular, $\rmCh_p(u)$ always represents as an integral functional}. Similarly, and thanks to locality \cite{AmbrosioGigliSavare11-3} of minimal $p$-weak upper gradients, we can introduce the space $W^{1,p}_{loc}(\X)$ as the space of $u \in L^p_{loc}(\mm)$ so that $\eta u \in W^{1,p}(\X)$ for all $\eta \in \Lip_{bs}(\X)$. Also, by slight abuse of notation, we shall always write $\|\nabla u\|_{L^p(\mm)}$ in place of $\| |\nabla u|\|_{L^p(\X)}$. We shall not insist on the dependence of $|\nabla u|$ on the exponent $p$, as in all the settings appearing in this note this is never an issue, see \cite{Cheeger00,GigliHan14,GigliNobili21}.

\medskip

In general, we cannot expect any sort of Sobolev embeddings to hold, as this genuinely depends on the underlying regularity of the ambient space. Since in this note we are interested in Sobolev inequalities, we start discussing notions of Ricci curvature lower bounds in this general setting.

The search for synthetic treatments of Ricci curvature lower bounds began with the Lott-Sturm-Villani theory \cite{Lott-Villani09,Sturm06I,Sturm06II} (see also \cite{Villani2016,Sturm24_Survey} and references therein) after the celebrated theory of Ricci limit spaces due to Cheeger-Colding \cite{Cheeger-Colding96,Cheeger-Colding97I,Cheeger-Colding97II,Cheeger-Colding97III}.  This notion is formulated via optimal transport (\cite{Villani09}) by requiring convexity of Entropy functionals in the Wasserstein space. In this note, we shall rather restrict our analysis to the \emph{Riemannian} subclass, the so-called, ${\sf RCD}(K,N)$ condition for $K\in\R, N\ge 1$, which is the requirement of a metric measure space $\Xdm$ to have ``Ricci curvature bounded below by $K$ and dimension bounded above by $N$''. For brevity reasons, we shall not introduce this notion while focusing instead on key aspects and examples. Here we just recall that the {\sf RCD}-condition was formulated first in the infinite-dimensional setting \cite{AmbrosioGigliSavare11-2} and then proposed in the finite-dimensional setting \cite{Gigli12} { to single out non-Riemannian structures. In particular, on an {\sf RCD}-space $\Xdm$, the functional
\[
W^{1,2}(\X) \ni u \mapsto \int |\nabla u|^2 \, \d \mm,
\]
satisfies the parallelogram identity, thus resembling the context of Riemannian manifolds.} We mention also additional important contributions \cite{AmbrosioGigliSavare12,AGMR15,AmbrosioMondinoSavare13-2,EKS15,CM16} that lead to many developments of this theory and thanks to which important equivalent formulations were identified. For a complete account, we refer to \cite{AmbICM}.

{Before giving some examples of ${\sf RCD}$-spaces and recall important properties, we recall a key ingredient on ${\sf RCD}(0,N)$ space $\Xdm$ that will be used several times.  In this setting, the Bishop-Gromov monotonocity holds (see \cite{Sturm06I,Sturm06II}): for all $x\in\X$ we have that
\[
r\mapsto \frac{\mm(B_r(x))}{\omega_Nr^N},\qquad\text{is non-increasing}.
\]
In particular, the following limit is well-defined and independent on $x$
\[
{\sf AVR}(\X) \coloneqq \lim_{R\to \infty} \frac{\mm(B_r(x))}{\omega_Nr^N} \in [0,\infty).
\]
As in the manifold case, when  ${\sf AVR}(\X)>0$, we say that $\X$ has Euclidean volume growth.}
\subsubsection{Examples}\label{sec:examples RCD} Let us discuss key examples of non-smooth spaces for this note. First of all, the {\sf RCD} class is consistent with the smooth category, and indeed it is a generalization of the class of Riemannian manifolds. Indeed, if $(M,g)$ is a $d$-dimensional Riemannian manifold with ${\sf Ric}_g\ge K$ and $d\le N$, then $(M,\sfd_g,{\rm Vol}_g)$  when regarded as a metric measure spaces with the geodesic distance and its volume measure, is an ${\sf RCD}(K,N)$ space. See \cite{SturmVonRenesse09} for several equivalences.

However, the class is much richer and contains also weighted Riemannian manifolds with modified Bakry-\'Emery Ricci tensor bounded below, and singular structures such as Ricci limit spaces  \cite{Cheeger-Colding96,Cheeger-Colding97I} and Alexandrov spaces (\cite{Petrunin11,ZhangZhu10,GKKO}). For us, it will be important to recall some operations that are closed within this class, namely scalings and cones.

Let $\Xdm$ be a metric measure space satisfying the ${\sf RCD}(0,N)$ condition for some $N\ge 1$ and let $\sigma>0$ be a number. Then, the metric measure space obtained by scaling
\[
(\X_\sigma,\sfd_\sigma,\mm_\sigma) \coloneqq (\X,\sigma\cdot \sfd,\sigma^N\cdot \mm),
\]
is again an ${\sf RCD}(0,N)$ space, where we set $\sigma \cdot \sfd (x,y) \coloneqq \sigma \sfd(x,y)$ for all $x,y \in \X$ and $\sigma^N\cdot \mm(B) = \sigma^N \mm(B)$ for every $B\subset \X$ Borel. 

For $N>1$, recall a that a $N$-Euclidean cone over a metric measure space $(Y,\rho,\mu)$ is defined  to be the space $C_N(Y)\coloneqq  Y\times [0,\infty)/(Y\times \{0\})$ endowed with the distance and the measure
\begin{align*}
    &\sfd((t,y),(s,y'))\coloneqq \sqrt{t^2+s^2-2st\cos(\rho(y,y')\wedge \pi)},\\
    &\mm\coloneqq t^{N-1} \d t\,\otimes\, \mu, 
\end{align*}
for all $t,s \in [0,\infty), y,y' \in Y$. The point $Y \times \{0\}$ is called the tip of the cone. It follows then by \cite{Ketterer13} (and thanks to \cite{CM16}) that the ${\sf RCD}$-condition is stable under the action of taking cones, namely, we have for $N>2$ that 
\[
C_N(Y) \text{ is an }{\sf RCD}(0,N)\text{-space}\qquad \iff \qquad Y\text{ is an }{\sf RCD}(N-2,N-1)\text{-space}.
\]
Similarly, we recall the notion of $N$-spherical suspension over a metric measure space $(Y,\mm_Y,\sfd_Y)$ defined as the space $S_N(Y) \coloneqq  Y\times \left[0,\pi \right]/(Y\times \left\{0,\pi\right\})$ endowed with the distance and the measure
\begin{align*}
    &\sfd((t,y),(s,y'))\coloneqq  \cos^{-1}\big(\cos(s)\cos(t)+\sin(s)\sin(t)\cos\left(\rho(y,y')\wedge\pi\right)\big),\\ 
    &\mm\coloneqq  \sin^{N-1}(t) \d t\,\otimes\, \mu.
\end{align*}
for all $t,s \in [0,\pi], y,y' \in Y$. The points $\{0\}\times Y,\{\pi\}\times Y$ are called tips of the suspension. Again by \cite{Ketterer13}, the ${\sf RCD}$-condition is stable under the action of taking suspension: if $N>2$ it holds
\[
S_N(Y) \text{ is an }{\sf RCD}(N-1,N)\text{-space}\qquad \iff \qquad Y\text{ is an }{\sf RCD}(N-2,N-1)\text{-space}.
\]

\subsubsection{Precompactness under Gromov-Hausdorff convergence}\label{sec:pmGH} One of the main reasons to consider synthetic notions of Ricci curvature lower bounds is to gain pre-compactness and stability properties under a weak notion of convergence of metric structures dating back to Gromov \cite{Gromov07}. Here, we shall adopt the so-called pointed measure Gromov-Hausdorff convergence of metric measure space following  \cite{GigliMondinoSavare13}. Even though this is not standard, it better suits the goal of this note and it is equivalent, in the case of finite dimensional ${\sf RCD}$-spaces, to previous notions considered in the literature (see again \cite{GigliMondinoSavare13}).

Set $\bar \N \coloneqq \N \cup\{\infty\}$. A \emph{pointed} metric measure space is a quadruple $(\X,\sfd,\mm,x)$ where $\Xdm$ is a metric measure space and $x \in \X$. A pointed ${\sf RCD}$-space is then a pointed metric measure space  $(\X,\sfd,\mm,x)$  so that $\Xdm$ is an ${\sf RCD}$-space.
\begin{definition}
Let $(\X_n,\sfd_n,\mm_n,x_n)$ be pointed metric measure spaces for $n \in \bar \N$. We say that $(\X_n,\sfd_n,\mm_n,x_n)$ converges to $(\X_\infty,\sfd_\infty,\mm_\infty,x_\infty)$ in the pointed measured Gromov Hausdorff topology, provided there are a metric space $(\Z,\sfd)$ and isometric embeddings $\iota_n \colon \X_n \to \Z$ for all $n \in \bar \N$ satisfying
\[
    (\iota_n)_\sharp \mm_n \weakto (\iota_\infty)_\sharp\mm_\infty,\qquad \text{in duality with }C_{bs}(\Z),
\]
and $\iota_n(x_n)\to \iota_\infty(x_\infty)$ as $n\uparrow\infty$. The metric space $(\Z,\sfd)$ is called the \emph{realization of the convergence}. In this case, we shortly say that $\X_n$ pmGH-converges to $\X_\infty$ and write $\X_n \overset{pmGH}{\to}\X_\infty$. 
\end{definition}

The key results are then the pre-compactness and the stability properties of the ${\sf RCD}$-condition, referring to \cite{Gigli10,AmbrosioGigliSavare11-2,GigliMondinoSavare13} (recall also \cite{Sturm06I,Sturm06II,LV09}) and thanks to Gromov's precompactness \cite{Gromov07}). 
\begin{theorem}\label{ref:precompact RCD}
    Let $(\X_n,\sfd_n,\mm_n,x_n)$ be pointed ${\sf RCD}(K_n,N_n)$ spaces for $n\in\N$ and for some $K_n \in \R,N_n \in [1,\infty)$ satisfying $K_n \to K \in \R, N_n \to N \in [1,\infty)$. Suppose that $\mm_n(B_1(x_n)) \in (v,v^{-1})$ for some $v>0$ independent on $n$. Then, there exist a pointed ${\sf RCD}(K,N)$ space $(\X_\infty,\sfd_\infty,\mm_\infty,x_\infty)$ and a subsenquence $(n_k)$ such that $\X_{n_k} \overset{pmGH}{\to}\X_\infty$ as $k\uparrow\infty$.
\end{theorem}
We also refer to \cite{Gigli23_working} for more on stability results in connection with Ricci curvature lower bounds.
\subsubsection{Sobolev calculus with varying base space}\label{sec:sob X varying}
Building on top of the pre-compactness and stability properties of the ${\sf RCD}$-condition, one can start looking at Sobolev functions $u_n \in W^{1,2}(\X_n)$ along a pmGH-converging sequence. The goal is to discuss semicontinuity and Rellich pre-compactness results for a generalized Sobolev calculus with varying base space. Here, the ${\sf RCD}$-condition will be essential, as the only chance to obtain such a result is to enforce uniform underlying regularity, see \cite{Gigli23_working} for a detailed treatment.

\medskip

We start recalling the notions of convergence of functions along pmGH-converging sequences \cite{Honda15,GigliMondinoSavare13,AmbrosioHonda17}, adopting the so-called \emph{extrinsic approach}, see \cite{GigliMondinoSavare13}.
\begin{definition}\label{def:lpconv}
 Let $(\X_n,\sfd_n,\mm_n,x_n)$ be pointed metric measure spaces for $n \in \bar \N$ and suppose that $\X_n \overset{pmGH}{\to}\X_\infty$ as $n\uparrow\infty$. Let $p \in (1,\infty)$ and fix a realization  of the convergence in $(\Z,\sfd)$. We say:
\begin{itemize}
\item[{\rm i)}] $f_n\in L^p(\mm_n)$ \emph{converges $L^p$-weak} to $f_\infty\in L^p(\mm_\infty)$, provided $\sup_{n \in \N}\|f_n\|_{L^p(\mm_n)}<\infty$ and $f_n\mm_n \weakto f_\infty\mm_\infty$ in duality with $C_{bs}(\Z)$;
\item[{\rm ii)}] $f_n\in L^p(\mm_n)$ \emph{converges $L^p$-strong} to $f_\infty\in L^p(\mm_\infty)$, provided it converges $L^p$-weak and $\limsup_n \|f_n\|_{L^p(\mm_n)} \le  \|f_\infty\|_{L^p(\mm_\infty)}$;
\item[{\rm iii)}] $f_n \in W^{1,p}(\X_n)$ \emph{converges $W^{1,p}$-weak} to $f_\infty \in L^p(\X_\infty)$ provided it converges $L^p$-weak and $\sup_{n \in \N} \| \nabla f_n \|_{L^p(\mm_n)}<\infty$;
    \item[{\rm iv)}] $f_n \in W^{1,p}(\X_n)$ \emph{converges $W^{1,p}$-strong} to $f_\infty \in W^{1,p}(\X_\infty)$ provided it converges $L^p$-strong and $\| \nabla f_n \|_{L^p(\mm_n)} \to \|\nabla f_\infty \|_{L^p(\mm_\infty)}$;
\item[{\rm v)}] $f_n\in L^p(\mm_n)$ \emph{converges $L^p_{loc}$-strong} to $f_\infty\in L^p(\mm_\infty)$, provided $\eta f_n$ converges $L^p$-strong to $\eta f_\infty$ for every $\eta \in C_{bs}(\Z)$.
\end{itemize}
\end{definition}
The first key stability result is the so-called Mosco-convergence of Cheeger energies proved in \cite{GigliMondinoSavare13} for possibly infinite dimensional spaces (see \cite{AmbrosioHonda17} for conclusion ii) and for the case $p\in(1,\infty)$).
\begin{theorem}[Mosco-convergence of $\rmCh_2$]\label{thm:mosco}
Let $(\X_n,\sfd_n,\mm_n,x_n)$ be pointed ${\sf RCD}({ K},N)$ space for some ${ K \in\R}, N \in [1,\infty)$ and for $n \in \bar \N$ and suppose that $\X_n \overset{pmGH}{\to}\X_\infty$ as $n\uparrow\infty$. Then:
\begin{itemize}
    \item[{\rm i)}]if $f_n\in L^2(\mm_n)$ converges $L^2$-weak to $f_\infty \in L^2(\mm_\infty)$, it holds
\begin{equation}
   \rmCh_2(f_\infty) \le \liminf_{n\to\infty} \rmCh_2(f_n); \label{eq:Gamma liminf}
\end{equation}
\item[{\rm ii)}]for any $f_\infty \in L^2(\mm_\infty)$, there exists $f_n \in L^p(\mm_n)$ converging $L^2$-strong to $f_\infty$ and it holds
\[  \limsup_{n\to\infty}  \rmCh_2(f_n) \le \rmCh_2(f_\infty).
\]
In particular, $f_n$ converges $W^{1,2}$-strong to $f_\infty$.
\end{itemize}
\end{theorem}
Result i) should be regarded as a lower semicontinuity result of the Sobolev space along varying base space. A straightforward consequence of this is that Sobolev inequalities are \emph{stable} along pmGH-convergences of metric measure structure, following \cite[Lemma 4.1]{NobiliViolo21}.
\begin{lemma}[pmGH-Stability of Sobolev constants]\label{lem:sobolev stability}
    Let $(\X_n,\sfd_n,\mm_n,x_n)$ be pointed ${\sf RCD}({ K},N)$ space for some ${ K\in \R}, N \in [1,\infty)$ and for $n \in \bar \N$ and assume that for some $A_n \to A>0$ and $1<q<\infty$, it holds 
    \[
        \|u\|_{L^q(\mm_n)} \le A_n\| \nabla u\|_{L^2(\mm_n)},\qquad \forall u \in W^{1,2}(\X_n).
    \]
    If $\X_n \overset{pmGH}{\to}\X_\infty$ as $n\uparrow\infty$, then it holds
    \[
        \|u\|_{L^q(\mm_\infty)} \le A\| \nabla u\|_{L^2(\mm_\infty)},\qquad \forall u \in W^{1,2}(\X_\infty).
    \]
\end{lemma}
\begin{proof}
		Fix any $u \in \LIP_{bs}(\X_\infty)$, and by ii) in Theorem \ref{thm:mosco} take a sequence $u_n \in W^{1,2}(\X_n)$ such that $u_n$ converges in $W^{1,2}$-strong to $u$. In particular, we can estimate by assumptions
	\[
		\limsup_{n\to \infty}\|u_n\|_{L^q(\mm_n)}
		\le \limsup_{n\to \infty} A \|\nabla u_n\|_{L^2(\mm_n)} = A \|\nabla u\|_{L^2(\mm_\infty)}<\infty.
	\]
	Therefore, by weak lower semicontinuity of the $L^q$-norm (in this setting, see \cite{GigliMondinoSavare13,AmbrosioHonda17}) we get the conclusion of the proof, also taking into account the arbitrariness of $u \in \LIP_{bs}(\X_\infty)$.
\end{proof}
\subsection{Fine P\'olya-Szeg\H{o} inequality and applications}\label{sec:Polya}
The second important ingredient for our analysis is a fine P\'olya-Szeg\H{o} rearrangement theory on non-smooth spaces. 

The genesis of rearrangement methods for functions dates back to the works of Faber \cite{Faber23} and Krahn \cite{Krahn25,Krahn26} and to the book \cite{PoliaSzego51} from which the so-called P\'olya-Szeg\H{o} inequality originated (see also the modern books \cite{LiebLoss1997,Baernstein19,ciao}, the survey \cite{Frank2022} and references therein). To study the case of equality in the P\'olya-Szeg\H{o} inequality, a big step has been made by Brothers-Ziemer  \cite{BrothersZiemer88} (see also \cite{FeroneVolpicelli03}). This fine investigation culminated in \cite{CianchiFusco02} where this analysis was extended to the case of functions of bounded variation. 

In metric setting, these principles were exported in \cite{MondinoSemola20} on non-smooth spaces having positive Ricci curvature lower bounds (extending previous work of Berard-Meyer \cite{BerardMeier1982}) and in \cite{NobiliViolo21,NobiliViolo24} with the primary focus on non-negative Ricci curvature lower bounds. Here, we shall follow the general analysis conducted in \cite{NobiliViolo24_PZ} to get finer rigidity results that are essential in this note.

\medskip

Let us consider a metric measure space $\Xdm$ and $u  \colon \X \to \R$ Borel satisfying
\[
\mu(t)\coloneqq \mm(\{u>t\})<\infty,\qquad \forall \, t>\essinf u.
\]
For $N>1$, consider the one-dimensional weighted measure on the half-line
\[
\mm_{0,N} \coloneqq   N\omega_N t^{N-1}\d t\mres{(0,\infty)},
\]
with $\omega_N\coloneqq\frac{\pi^{N/2}}{\Gamma\left(N/2+1\right)}$. If $N\in \N$ then $\omega_N$ is the volume of the unit ball in $\rr^N$ as already introduced. We can define the decreasing rearrangement of $u$ with respect to $\mm_{0,N} $ as the function
\[
(0,\infty)\ni x \mapsto u^*(x) \coloneqq \inf\{ t >\essinf u \, \colon \, \mu(t)< \mm_{0,N}((0,x)) \}.
\]
Observe that, by definition, $u$ and $u^*$ are equi-measurable and $u^*$ is monotone non-increasing, left-continuous and finitely valued. Hence, $u^*$ is a.e.\ differentiable in $(0,\infty)$. Moreover, $u^*$ is independent of the representative of $u$ up to $\mm$-a.e.\ equality. Our main result in this part is reported (in a simplified form) from \cite[Theorem 1.3]{NobiliViolo24_PZ}.
\begin{theorem}\label{thm:polya noncompact}
 Let $\Xdm$ be a ${\sf RCD}(0,N)$ space, with $ N \in (1,\infty)$ and such that ${\sf AVR }(\X)>0$. Fix also $p \in (1,\infty)$. Then, for every $u \in W^{1,p}_{loc}(\X)$ with $\mm(\{ u>t \})<\infty$ for all $t >\essinf u$, it holds
    \begin{equation}
        \int |\nabla u|^p\,\d \mm \ge  {\sf AVR}(\X)^{\frac pN} \int_0^{+\infty} |(u^*)'|^p\, \d \mm_{0,N}, \label{eq:polya Sobolev AVR}
    \end{equation}
 meaning that, when the left-hand side is finite, then $u^*$ is locally absolutely continuous and the above holds.

Moreover, if equality occurs in \eqref{eq:polya Sobolev AVR} { for some  $0\neq u \in W^{1,p}_{loc}(\X)$ with $\mm(\{ u>t \})<\infty$ for all $t >\essinf u$} and with both sides non-zero and finite, then $\X$ is an  $N$-Euclidean cone. Finally, if also $(u^*)'\neq 0$ a.e.\ on $\{\essinf u< u^* < \esssup u\}$ then $u$ is radial, i.e.\ for some tip $z_0\in \X$ it holds
\[
    u = u^* \circ\big( {\sf AVR}(\X)^{\frac 1N} \sfd(\cdot,z_0)\big),\qquad  \mm\text{-a.e.}.
\]
\end{theorem}
For non-negative functions, \eqref{eq:polya Sobolev AVR}  was already studied in the non-smooth setting in \cite{NobiliViolo21,NobiliViolo24} and previously on manifolds in \cite{BaloghKristaly21}. It was obtained by combination of a well-established isoperimetric inequality in this setting, see \cite{BaloghKristaly21} and previous contributions on manifolds \cite{Brendle20,AgostinianiFogagnoloMazzieri20,FogagnoloMazzieri22}, with rearrangements arguments involving the theory of sets of finite perimeter in metric setting \cite{Miranda03,AmbrosioDiMarino14} (see also \cite{Martio16}, shown to be equivalent in \cite{DCEBKS19,NobiliPasqualettoSchultz21}, we also refer to \cite{BrenaNobiliPasqualetto22} for a metric-valued theory).

Let us focus on the second part of the statement concerning equality cases. In \cite{NobiliViolo24}, we also dealt with equality cases relying on the rigidity part of the underlying isoperimetric inequality \cite{AntonelliPasqualettoPozzettaSemola22,CavallettiManini22} and on the boundedness regularity properties of isoperimetric sets \cite{AntonelliPasqualettoPozzetta22,APPV23}. However, since  \eqref{eq:polya Sobolev AVR} was deduced by approximation via Lipschitz functions satisfying $|\nabla u|\neq 0$ $\mm$-a.e.\ on $\{\essinf u<u<\esssup u\}$, the radiality conclusion could only be deduced under these assumptions. The main novelty in the above is that the same conclusion holds under $(u^*)'\neq 0$ a.e.\ in $\{\essinf u< u^*<\esssup u\}$ (and completely dropping the Lipschitz regularity of $u$). It is well known (see, e.g., \cite[Proposition 3.9]{NobiliViolo24_PZ}) that the latter is {apriori} a \emph{weaker} assumption 
\begin{align*}
    \mm( \{ |\nabla u|=0\} \cap &\{ \essinf u<u<\esssup u\})=0  \\
    &\implies \, \mm_{0,N}( \{ (u^*)'=0\} \cap \{ \essinf u<u^*<\esssup u\})=0,
\end{align*}
{and that it can be strictly weaker (see \cite[Example 3.6]{CianchiFusco02}).} Therefore, Theorem \ref{thm:polya noncompact} improves the corresponding statement in \cite[Theorem 3.4]{NobiliViolo24}. {However, it is also relevant to say that, when equality occurs in \eqref{eq:polya Sobolev AVR} for some non-zero function, the two conditions turn out to be equivalent (see, e.g., \cite[Proposition 3.9]{NobiliViolo24_PZ}).}

A straightforward implication is then Theorem \ref{thm:rigidity sob main avr}, also in non-smooth settings (here, $S_{N,p}$ reads exactly as in \eqref{eq:Sobolev constant} for possibly non-integer dimension).
\begin{theorem}\label{thm:sob avr rcd}
    Let $p\in(1,\infty)$ and let $\Xdm$ be a ${\sf RCD}(0,N)$ space for some $ N \in (p,\infty)$ and with ${\sf AVR }(\X)>0$. Set $p^*\coloneqq pN/(N-p)$. Then, for every $u \in W^{1,p}_{loc}(\X)\cap L^{p^*}(\X)$, it holds sharp
    \begin{equation}
        \|u\|_{L^{p*}(\mm)} \le {\sf AVR}(\X)^{-\frac 1N} S_{N,p}\|\nabla u\|_{L^p(\mm)}.\label{eq:sob avr RCD}
    \end{equation}
    Moreover, if equality holds for some $0\neq u \in W^{1,p}_{loc}(\X)\cap L^{p^*}(\X)$ with both sides finite, then
    \begin{itemize}
        \item[{\rm i)}] $\X$ is an $N$-Euclidean cone;
        \item[{\rm ii)}] there are $a\in \R,b>0$ and a tip $z_0 \in \X$ so that
        \[
        u = \frac{a}{(1+b\sfd(\cdot,z_0)^{\frac{p}{p-1}})^\frac{N-p}{p}},\qquad \mm\text{-a.e..}
        \]
    \end{itemize}
\end{theorem}
\begin{proof}
        We first show that \eqref{eq:sob avr RCD} holds. Consider any $u \in W^{1,p}_{loc}(\X)\cap L^{p^*}(\mm)$. If $u=0$ or $\|\nabla u\|_{L^p(\mm)}=\infty$, there is nothing to prove. Otherwise, consider $|u|^*$ the decreasing rearrangement of $|u|$ with respect to  $\mm_{0,N}$ (notice $\mm(|u|>t)<\infty$ for $t>0$, as $u \in L^{p^*}(\mm)$).   Recalling Theorem \ref{thm:polya noncompact}, we have that $|u|^*$ is locally absolutely continuous on $(0,\infty)$, hence we can estimate
        \[
             S_{N,p}{\sf AVR}(\X)^{-\frac 1N}\|\nabla u\|_{L^p(\mm)}\overset{\eqref{eq:polya Sobolev AVR}}{\ge}  S_{N,p} \|(|u|^*)'\|_{L^{p^*}(\mm_{0,N})} \ge    \||u|^*\|_{L^{p^*}(\mm_{0,N})} =\|u\|_{L^{p^*}(\mm_{0,N})} ,
        \]
        having used that $|\nabla u|=|\nabla |u||$ $\mm$-a.e.\ by the chain rule, Bliss inequality \cite{Bliss30} in the second inequality and equimeasurability in the last equality. This proves \eqref{eq:sob avr RCD}.

        Assume now that equality holds in \eqref{eq:sob avr RCD} for some function $0\neq u \in W^{1,p}_{loc}(\X)\cap L^{p^*}(\mm)$ with both sides non-zero and finite. Then, in the above computation, we have that all inequalities must be equalities. By the characterizations of extremal functions in the Bliss inequality, we then get
    \[
     |u|^*(t) = a\big(1+b t^{\frac p{p-1}}\big)^{\frac{p-N}{p}},\qquad\forall \, t>0,
    \]
    for some $a\in \R,b>0$. In particular, it holds $(|u|^*)'\neq 0$ a.e.\ on $(0,\infty)$. Therefore, the rigidity part in Theorem \ref{thm:polya noncompact} gives that $\X$ is an $N$-Euclidean cone and $|u|= |u|^*\circ \big({\sf AVR}(\X)^{1/N} \sfd(\cdot,z_0)\big)$ holds for some tip $z_0\in\X$. Since $|u|^*$ is strictly positive,  the claim for $u$ follows up to possibly inverting sign.

    Finally, the sharpness of \eqref{eq:sob avr RCD} follows {by a contradiction and limiting argument (see Remark \ref{rem:sharp} below for the relevant references). Indeed, suppose that there is $A< {\sf AVR}(\X)^{-\frac 1N}S_{N,p}$ so that 
    \begin{equation}
        \|u\|_{L^{p^*}(\mm)} \le A \|\nabla u\|_{L^p(\mm)},
    \label{eq:sob c fake}
    \end{equation}
    holds (i.e.\ there is a better Sobolev constant in \eqref{eq:sob avr RCD}). Consider now $x\in\X$ and $\sigma_n \downarrow 0$ and set
    \[
    (Y_n,\rho_n,\mu_n,y_n) \coloneqq (\X,\sigma_n \cdot \sfd, \sigma_n^N \cdot \mm,x),\qquad \forall n \in \N.
    \]
    For all $r>0$, we know by scaling that $\frac{\mu_n(B_r(y_n))}{\omega_N r^N} = \frac{\sigma_n^N \mm (B_{r/\sigma_n}(x))}{\omega_N r^N} = \frac{ \mm (B_{r/\sigma_n}(x))}{\omega_N (r/\sigma_n)^{N}}$. Therefore, the sequence of pointed metric measure space $(Y_n,\rho_n,\mu_n,y_n) $ satisfies the ${\sf RCD}(0,N)$ condition and, by Bishop-Gromov monotonicity, we deduce
    \[
    \mu_n(B_1(y_n)) \in ( {\sf AVR}(\X), \omega_N \mm(B_1(x))),
    \]
    for all $n$ sufficiently large. In particular, Theorem \ref{ref:precompact RCD} applies and, up to a subsequence, we get the existence of a pointed ${\sf RCD}(0,N)$ metric measure space $(Y,\rho,\mu,z)$ so that $Y_n \overset{pmGH}{\to} Y$ as $n\uparrow\infty$. Since it holds that $\mu_n(B_r(y_n)) \to \mu(B_r(z))$ (see e.g.\ \cite[Eq. (2.4)]{DePhilippisGigli18}) and, by scaling, that $\mu_n(B_r(y_n)) /\omega_Nr^N \to {\sf AVR}(\X)$ as $n\uparrow\infty$, we get in particular the identities
    \[
    {\sf AVR}(\X) = \frac{\mu(B_r(z))}{\omega_Nr^N},\quad \forall r>0,\qquad \text{hence also}\qquad {\sf AVR}(\X)={\sf AVR}(Y).
    \]
    In particular, by the volume-cone-to-metric-cone theorem \cite{DePhilippisGigli15} we deduce that $Y$ is an $N$-Euclidean cone and $z \in Y$ is a tip (the space $Y$ obtained is usually called an asymptotic cone or blow down of $\X$ at $x$). Here comes the main point. On the $N$-Euclidean cone $Y$, the bubble $u = (1+\rho(\cdot,z)^{\frac{p}{p-1}})^{\frac{p-N}{p}}$ is a minimizer for the Sobolev inequality \eqref{eq:sob avr RCD}. This easily follows using the characterization \cite{Bliss30} and by polar coordinates around the tip $z$ (see, e.g., \cite[Lemma 4.2]{NobiliViolo21}). Hence we have
    \[
        \frac{\|u\|_{L^{p^*}(\mu)}}{\|\nabla u\|_{L^p(\mu)}} = {\sf AVR}(Y)^{-\frac 1N}S_{N,p}  =  {\sf AVR}(\X)^{-\frac 1N}S_{N,p}.
    \]
    Notice, however, that the sequence of spaces $(Y_n,\rho_n,\mu_n)$ satisfies the same Sobolev inequality as in \eqref{eq:sob c fake} by invariance of the inequality under the chosen scalings. Therefore, the stability of Sobolev constants (c.f.\ \cite[Lemma 4.1]{NobiliViolo21}) yields
    \[
        \frac{\|u\|_{L^{p^*}(\mu)}}{\|\nabla u\|_{L^p(\mu)}} \le A.
    \]
    We thus found a contradiction with the absurd hypothesis $A< {\sf AVR}(\X)^{-\frac 1N}S_{N,p}$.}
\end{proof}
We mention that the very same arguments employed in the above proof works also for different functional inequalities, see for instance \cite[Theorem 1.6]{NobiliViolo24_PZ} and reference there for what concerns the Faberk-Krahn inequality and the log-Sobolev inequality in this setting.
\begin{remark}\label{rem:sharp}
{The argument to deduce sharpness of \eqref{eq:sob avr RCD} have its roots in the works \cite{Ledoux99,Xia01}} where it was understood that the validity of a Sobolev inequality as in \eqref{eq:sob c fake} under non-negative Ricci curvature gives, in turn, the non-collapsing property
\[
{\sf AVR}(\X) \ge \left( \frac{S_{N,p}}{A} \right)^N.
\]
We considerably shorten the proof of this result via asymptotic analysis and non-smooth techniques in \cite[Theorem 4.6]{NobiliViolo21}. {The argument was presented in the above proof as it implies sharpeness of \eqref{eq:sob avr RCD}}. Finally, we also refer to the recent work \cite{Kristaly23} containing the shortest and most elementary argument available by now (see \cite[Proposition 2.3]{Kristaly23}).  
\end{remark}

\begin{remark}\label{rmk:regularity PDE}
\rm As already discussed, for $p=2$ the results i),ii) in Theorem \ref{thm:sob avr rcd} were already obtained in \cite{NobiliViolo24} relying on a weaker rigidity result of the P\'olya-Szeg\H{o} inequality. Notice, in the above proof, that the property 
\[
\mm_{0,N}( \{ (u^*)'=0\} \cap \{ \essinf u<u^*<\esssup u\})=0,
\]
comes for free, thanks to the absence of critical points of the extremal functions in the Bliss inequality. Instead, in \cite[Theorem 5.3]{NobiliViolo24} the property $\mm( \{ |\nabla u|=0\} \cap \{ \essinf u<u<\esssup u\})=0$ had to be checked, together with the fact that equality in \eqref{eq:sob avr RCD} holds for Lipschitz functions. Both pieces of information were extracted from elliptic regularity theory thanks to the fact that an extremal function in the Sobolev inequality is a signed distributional solution of
\[
S_{N,2}^2{\sf AVR}(\X)^{-\frac 2N}\Delta u +u|u|^{2^*-2}=0,
\]
for an appropriate notion of Laplacian in this setting (see \cite{Gigli12}). The above has a non-linear analogue for $p\neq 2$, but the arguments used in \cite[Theorem 5.3]{NobiliViolo24} would not directly extend in this case.  \fr
\end{remark}

\subsection{Concentration compactness on varying spaces}
Here we present the last important ingredient, namely a generalized concentration compactness principle for extremizing sequences of Sobolev inequalities with varying base space, revisiting the works of Lions \cite{Lions84,Lions85}. As already remarked, we focus only on the setting of Theorem \ref{thm:qualitative SobAVR intro} for simplicity, as it contains all the ideas that, suitably generalized, work also for Theorem \ref{thm:into main compact}. We therefore report here \cite[Theorem 6.2]{NobiliViolo24} in simplified form.
\begin{theorem}\label{thm:CC_Sobextremals}
   For every $N\in (2,\infty)$,  there exists $\eta_N\in(0,1/2)$ such that the following holds. Let $(Y_n,\rho_n,\mu_n,y_n)$ be pointed ${\sf RCD}(0,N)$ spaces for some $N\in (2,\infty)$ and for all $n\in\N$. Suppose that $\sup_n \mu_n(B_1(y_n))<+\infty$ and, for some $A_n\to A >0$ and $2^*=2N/(N-2)$, that it holds
   \begin{equation}
   \|u\|_{L^{2^*}(\mu_n)} \le A_n\|\nabla u\|_{L^2(\mu_n)},\qquad \forall u \in W^{1,2}(Y_n).\label{eq:convention}
   \end{equation}
Furthermore, suppose there are non-constant functions $u_n \in W^{1,2}(Y_n)$ with $\| u_n\|_{L^{2^*}(\mu_n)} =1$  and
\begin{align}
&\sup_{y \in Y_n} \int_{B_1(y)}|u_n|^{2^*}\,\d\mu_n=\int_{B_1(y_n)}|u_n|^{2^*}\,\d\mu_n = 1-\eta,\label{eq:Levyscalings} \\
 &\| u_n\|_{L^{2^*}(\mu_n)} \ge  \tilde A_n \|\nabla u_n\|_{L^2(\mu_n)},\label{eq:extremals}
\end{align}
for  some $\tilde A_n \to A$  and  $\eta \in(0,\eta_N).$ Then, up to a subsequence, we have:
    \begin{itemize}
    \item[ \rm i)]
    there is a pointed $\RCD(0,N)$-space $(Y,\rho,\mu,y)$ so that 
    \[
    Y_n\overset{pmGH}{\to} Y,
    \]
    and it holds
    \begin{equation}
       \|u\|_{L^{2^*}(\mu)} \le A\|\nabla u\|_{L^2(\mu)},\qquad \forall u \in W^{1,2}(Y);\label{eq:sob limit Y}
       \end{equation}
    \item[ \rm ii)]  $u_n$ converges  $L^{2^*}$-strong to some $0\neq u_\infty \in W^{1,2}_{loc}(Y)$ with $|\nabla u_\infty| \in L^2(\mu)$ and
    \[
       \int |\nabla u_n|^2\, \d \mu_n \to  \int |\nabla u_\infty|^2\,\d\mu, \qquad  \text{as }n\uparrow\infty;
    \]
    \item[ \rm iii)] it holds
    \[ \| u_\infty\|_{L^{2^*}(\mu)} =  A \|\nabla u_\infty\|_{L^2(\mu)}.\]

\end{itemize}
\end{theorem}
Next, we briefly comment on its proof while referring again to \cite[Theorem 6.2]{NobiliViolo24} for all the details. Before continuing, notice that the assumption
\[
\sup_n \mu_n(B_1(y_n))<\infty,
\]
is about the non-smooth setting. Indeed, if $Y_n$ is a sequence of manifolds and $\rho_n,\mu_n$ are respectively the geodesic distance and the volume measure, this is automatically true by Bishop-Gromov monotonicity.
\begin{proof}[Proof of i)]
    The first important key step is to show that the sequence $(Y_n,\rho_n,\mu_n,y_n)$ is pmGH-precompact thanks to the application of Gromov precompactness theorem in this setting, recall Theorem \ref{ref:precompact RCD}. To this aim, we need to check that
\[
\mu_n(B_1(y_n)) \in (v^{-1},v),\qquad \text{for some }v>0\text{ and for all }n\in\N.
\]
The upper bound is directly given by the assumptions. For the lower bound, it suffices to plug in \eqref{eq:convention} the sequence of $1$-Lipschitz cut-offs $\varphi_n \in \Lip(Y_n)$ so that $\varphi_n$ takes value in $[0,1]$,  $\varphi_n \equiv 1$ on $B_1(y_n)$ and $\varphi_n(x)=0$ on $B_2(y_n)^c$. Indeed, this gives
\[
    \mu_n(B_1(y_n))^{\frac{1}{2^*}} \le A_n \mu_n(B_2(y_n))^{\frac 12} \le C_N A_n \mu_n(B_1(y_n))^{\frac 12},
\]
where $C_N>0$ depends on the doubling constant of the ${\sf RCD}(0,N)$ class which is uniform along the sequence of spaces $Y_n$ { (see \cite[Corollary 2.4]{Sturm06II})}. In particular, rearranging the above gives
\[
\mu_n(B_{1}(y_n))^{\frac 12 - \frac{1}{2^*}}  \ge C_NA_n,
\]
which precisely implies the sought lower bound, as $2^*>2$ and $A_n \to A$ by assumptions. Therefore, by Gromov precompactness result in  Theorem \ref{ref:precompact RCD}, we have the existence of a pointed  ${\sf RCD}(0,N)$ space $(Y,\rho,\mu,y)$ so that, up to a subsequence, it holds $Y_n \overset{pmGH}{\to} Y$ as $n\uparrow\infty$. Moreover, by stability of Sobolev constants (c.f.\ Lemma \ref{lem:sobolev stability}), the proof of i) is concluded. 
\end{proof}
Up to now, the condition \eqref{eq:extremals} as well as the property \eqref{eq:Levyscalings} were not used at all. Roughly, the former encodes the fact that $(u_n)$ is an extremizing sequence along varying base space, while, the latter, states that precisely at unitary radius/scale around $y_n$, we detect definitely $1-\eta$ mass (recall that $|u_n|^{2^*}\mu_n$ are probabilities) with $1-\eta\neq 0,1$. The combination of the two with the generalized Sobolev calculus discussed in Section \ref{sec:sob X varying}, will put the concentration compactness strategy into work to get results ii) and iii). This is one of the most technical parts of the works \cite{NobiliViolo21,NobiliViolo24}, hence we will not give a rigorous proof. However, we try to convey the key ideas in the following overview.
\begin{proof}[Overwiew of the proof of ii) and iii)]
In this part, let us consider a realization $(\Z,\sfd)$ of the pmGH-convergence $Y_n\overset{pmGH}{\to}Y$ (thanks to result i), now being available). It can be shown that $(\Z,\sfd)$ can be taken proper, i.e.\ bounded and closed sets are compact, see \cite{GigliMondinoSavare13}. We can then consider the sequence of probability measures
\[
\nu_n \coloneqq |u_n|^{2^*}\mu_n.
\]
Here, we shall adopt the so-called \emph{extrinsic approach} (\cite{GigliMondinoSavare13}) and, by suitably pushing forward $\nu_n$ on $\Z$ via the isometric embeddings given by the pmGH-convergence, we shall tacitly consider $\mu_n$ as probability measures on $(\Z,\sfd)$.

Here comes the first main point of the concentration compactness, namely the analysis of all the possible behaviors of the sequence $(\nu_n)$. It turns out (see \cite[Lemma A.6]{NobiliViolo24}) that, up to a subsequence, \underline{one of the following} holds:
\begin{itemize}
    \item {\sc Compactness}. There exists $(z_n) \subset \Z$ such that for all $\eps>0$, there exists $R>0$ satisfying \[ \nu_n(B_R(z_n)) \ge 1-\eps, \qquad \forall n \in \N.\]

    \item {\sc Vanishing}. \[ \lim_{n\to\infty} \sup_{z \in \Z} \nu_n(B_R(z)) =0,\qquad \forall R>0.\]
    
    \item {\sc Dichotomy}. There exists $\lambda \in (0,1)$ with $\lambda \ge \limsup_n  \sup_{z \in \Z} \nu_n(B_R(z))$, for all $R>0,$ so that: there exists $R_n\uparrow \infty$, $(z_n) \subset \Z$ and there are $\nu_n^1,\nu_n^2$ two non-negative Borel measures satisfying
    \[
    \begin{split}
        &0 \le \nu_n^1+\nu_n^2 \le \nu_n, \\
        &\supp( \nu_n^1) \subset B_{R_n}(z_n), \quad \supp( \nu_n^2) \subset \Z \setminus B_{10R_n}(z_n), \\
        & \limsup_{n\to\infty}   \ \big| \lambda -  \nu_n^1( \Z) \big|+ \big| (1-\lambda ) -  \nu_n^2( \Z) \big|  =0 .
    \end{split}
    \]
\end{itemize}
The Compactness scenario exactly means that $(\nu_n)$ is tight, as a sequence of probability measures. The Vanishing scenario means that all the mass escapes outside of every ball and at every fixed scale/radius. The third and more complicated scenario is called instead Dichotomy as it refers to the fact that the mass splits into two separated regions, one located around the points $z_n$, the other escaping. From here, the proof is devoted to showing that the combination of \eqref{eq:Levyscalings} and \eqref{eq:extremals} rules out the Vanishing and Dichotomy scenario. Therefore, compactness will hold with $z_n=y_n$ as given by the assumptions, and results ii) and iii) and the conclusion of the proof will follow. 

\medskip 

Notice that \eqref{eq:Levyscalings} automatically excludes the Vanishing scenario. Hence, it remains to show that Dichotomy does not occur. { We argue in a simplified setting assuming that $\nu_n = \nu^1_n + \nu^2_n$, i.e.\ when all the mass is located in two disjoint regions at positive distance. The general case can be treated similarly by suitably handling remainder terms (see Step 2 in \cite[Theorem 6.2]{NobiliViolo24}). In this case,  one can decompose $u_n = u_n^1+u_n^2 \coloneqq  \phi_n u_n + (1-\phi_n)u_n$ for suitable Lipschitz cut-off functions $\phi_n \colon \Z\to [0,1]$. We thus deduce, using locality and that $(u_n)$ is an optimizing sequence, that
\begin{align*}
    1 = \|u_n\|_{L^{2^*}(\mu_n)}^2 &\overset{\eqref{eq:extremals}}{\ge} \limsup_{n\to\infty} A_n^2 \| \nabla u_n\|_{L^2(\mu_n)}^2 = \limsup_{n\to\infty} A_n^2 \| \nabla u_n^1 \|_{L^2(\mu_n)}^2  + A_n^2 \| \nabla u_n^2 \|_{L^2(\mu_n)}^2 \\
    &\overset{\eqref{eq:convention}}{\ge} \limsup_{n\to\infty} \|u_n^1\|_{L^{2^*}(\mu_n)}^2 + \|u_n^2\|_{L^{2^*}(\mu_n)}^2 \\
    &= \limsup_{n\to\infty} \left( \nu_n^1(\Z)\right)^{2/2^*} + \left(\nu_n^2(\Z)\right)^{2/2^*} \\
    &= \lambda^{2/2^*} + (1-\lambda)^{2/2^*} >1,
\end{align*}
having used, at last, the strict concavity of $t\mapsto t^{2/2^*}$ and that $\lambda \in (0,1)$. This leads to a contradiction, showing that} the Dichotomy scenario does not occur.

Therefore, having excluded all but the Compactness scenario, we get up to subsequence that $\nu_n \weakto \nu$ in duality with $C_b(Z)$ to some $\nu \in \PP(\Z)$. What is left to do, is to understand the nature of this convergence of probability measure, and how it is reflected at the level of $L^{2^*}$-strong convergence of the original functions $u_n$. Notice that, by standard semicontinuity and Rellich-compactness arguments (see \cite[Appendix A]{NobiliParise24}) we always have that $u_n$ converges (up to a subsequence) $L^{2^*}$-weak and $L^2_{loc}$-strong to some function $u_\infty \in W^{1,2}_{loc}(Y)$ with $|\nabla u_\infty|\in L^2(\mu)$, as well as that
\[
|\nabla u_n|^2\mu_n \weakto \omega, \qquad \text{in duality with }C_{bs}(\Z),
\]
as $n\uparrow\infty$ for some finite Borel measure $\omega$ (since $\|\nabla u_n\|_{L^2(\mu_n)}$ is uniformly bounded, by \eqref{eq:extremals}). The problem here is that $u_\infty$, for now, might be very well identically zero, while $\nu_n $ might be weakly converging to singular objects, such as a Dirac mass, in the sense of probability measures. Also, a combination of this phenomena can occur, and in fact, a complete characterization of the limit measures $\nu,\omega$ can be given in terms of $u_\infty$ and possibly countable many Dirac masses (see \cite[Lemma A.7]{NobiliViolo24}): there is a countable set of indices $J$, points $(x_j)_{j \in J}\subseteq Y$ and weights $(\mu_j)_{j\in J},(\nu_j)_{j \in J}\subseteq \mathbb R^+$, satisfying
    \begin{align*}
        &\nu = |u_\infty|^{2^*}\mu + \sum_{j \in J}\nu_j \delta_{x_j}, \\
        &\omega \ge |\nabla u|^2\omega + \sum_{j \in J}\omega_j\delta_{x_j},\\
        & \nu_j^{2/2^*} \le A^2 \omega_j,\qquad \sum_{j \in J}\nu_j^{2/2^*} <\infty.
    \end{align*}
Notice also that  $\lim_n \|\nabla   u_n \|^2_{L^2(\mu_n)} \ge \omega(\Z)$ by weak lower semicontinuity.  Therefore, by combining everything, we can finally produce the key chain of estimates in the concentration compactness method:
    \begin{align*}
        1 =  \lim_{n\to\infty} \int |u_n|^{2^*}\, \d \mu_n  &\ge \liminf_{n\to\infty}  \tilde  A_n^2\| \nabla  u_n \|^2_{L^2(\mu_n)}  \\
               &\ge  A^2 \omega(\Z) \ge  A^2\int |\nabla  u_\infty|^2\,\d \mu + \sum_{j\in J} \nu_j^{2/2^*} \\
          &\overset{\eqref{eq:sob limit Y}}{\ge} \Big(\int |u_\infty|^{2^*}\, \d \mu\Big)^{2/2^*} + \sum_{j\in J} \nu_j^{2/2^*} \\
           &\ge \Big( \int |u_\infty|^{2^*}\, \d \mu  + \sum_{j\in J} \nu_j \Big)^{2/2^*} = \nu(Y)^{2/2^*} = 1,
    \end{align*}
    where we used, in the last inequality, the concavity of the function $t^{2/2^*}$. In particular, all the inequalities must be equalities and, since $t^{2/2^*}$ is strictly concave, we infer that every term in the sum $\int |u_\infty|^{2^*}\, \d \mu  + \sum_{j\in J} \nu_j^{2/2^*}$ must vanish except one. By the assumption \eqref{eq:Levyscalings} and since $|u_\infty|^{2^*}\mm_n \rightharpoonup \nu$ in $C_b(\Z),$ we have $\nu_j\le 1-\eta$ for every $j \in J$. Hence $\nu_j=0$ and $\|u_\infty\|_{L^{2^*}(\mu)}=1$.
    This means that $u_n$ converges $L^{2^*}$-strong to $u$. Moreover, retracing the equalities in the above we have that $\lim_n \int |\nabla u_n|^2\, \d \mu_n = \int |\nabla u_\infty|^2\,\d\mu$. This proves point ii). Finally, equality in the fourth inequality is precisely part iii) of the statement.
\end{proof}
It is worth mentioning that related principles have been investigated also in the series of works \cite{AntonelliFogagnoloPozzetta21,AntonelliNardulliPozzetta22} (based on \cite{Nardulli14}) and applied to the study of the isoperimetric problem on non-compact manifolds in \cite{AntonelliBrueFogagnoloPozzetta22,AntonelliPasqualettoPozzettaSemola22,AntonelliPasqualettoPozzettaSemola23,AntonelliPozzetta23}. We refer to \cite{Pozzetta_survey} for a nice overview. Other recent applications of concentration compactness principles for variational problems can be found in the theory of clusters \cite{Resende23,NovagaPaoliniStepanovTortorelli23,NovagaPaoliniStepanovTortorelli22} and in the study of isoperimetric partitions \cite{CesaroniNovaga22,CesaroniFragalaNovaga23,CesaroniNovaga23_1,NovagaPaoliniTortorelli23,NobiliNovaga24}.
\section{The main argument for the stability}\label{sec:sketch}
In this part, we combine the three ingredients presented in the previous section and exhibit the main proof argument to get stability results for Sobolev inequalities under Ricci lower bounds.
\subsection{Proof in the non-compact case}
Here, we deal with the non-compact case.
\begin{proof}[Proof of Theorem \ref{thm:qualitative SobAVR intro}]
We can clearly assume that $\| u\|_{L^{2^*}(M)}=1$. Moreover, by an approximation argument, it is also sufficient to prove the statement for $u \in W^{1,2}(M)$.

We proceed by contradiction and suppose that there exist $\eps>0$, a sequence of $d$-dimensioanl Riemannian manifolds $(M_n,g_n)$ for $n\in\N$ satisfying ${\sf AVR}(M_n)>V$ and there are non-zero functions $u_n \in W^{1,2} \cap L^{2^*}(M_n)$ satisfying
\begin{equation}
    \| u_n\|_{L^{2^*}(M_n)} \ge (A_n-1/n) \| \nabla u_n\|_{L^2(M_n)}, \label{eq:AVR sobolev proof}
\end{equation}
where $A_n\coloneqq  {\sf AVR}(M_n)^{-\frac 1d}S_{d,2}$, but, by the contradiction hypothesis, so that 
\begin{equation}
        \inf_{a,b,z} \frac{\| \nabla ( u_n  -u_{a,b,z}  )\|_{L^2(M_n)}}{\|\nabla u_n\|_{L^2(M_n)}} > \eps,\qquad \forall n \in \N.\label{eq:contradiction AVR}
\end{equation}
For every $\eta \in (0,1)$ and $n\in\N$, let  $y_n \in M_n$ and $t_n>0$ so that
\[ 
1-\eta = \int_{B_{t_n}(y_n)} |u_n|^{2^*}\, \d {\rm Vol}_{g_n} =  \sup_{y \in M_n} \int_{B_{t_n}(y)} |u_n|^{2^*}\, \d {\rm Vol}_{g_n}.
\]
Define now $\sigma_n \coloneqq  t_n^{-1}$ and consider the sequence of pointed metric measure space $(Y_n,\rho_n,\mu_n,y_n) \coloneqq  (M_n, \sfd_{\sigma_n},\mm_{\sigma_n},y_n)$, where recall $\sfd_{\sigma_n}\coloneqq \sigma_n \cdot\sfd_{g_n}$, $\mm_{\sigma_n}\coloneqq \sigma_n^d \cdot{\rm Vol}_{g_n}$ and consider the functions $u_{\sigma_n}\coloneqq \sigma_n^{-d/2^*}u_n \in W^{1,2}(Y_n)$. By the compatibility properties of the ${\sf RCD}$-class, it is evident that
\[
(Y_n,\rho_n,\mu_n)\qquad\text{is an }{\sf RCD}(0,d)\text{ space},
\]
and, by definition, it holds
\begin{equation}
    \frac{\mu_n( B^{Y_n}_R(y))}{\omega_d R^d} =\frac{ {\rm Vol}_{g_n}( B_{ R/\sigma_n}(y))}{\omega_d (R/\sigma_n)^d}\label{eq:scaling}
\end{equation}
for all $R>0,n\in\N$. In particular, {by Bishop-Gromov monotonicity} we deduce the uniform bounds
\begin{align*}
    &\mu_n(B^{Y_n}_1(y_n)) \in {(V,\omega_d]},\\
    &{\sf AVR}(Y_n)  ={\sf AVR}(M_n) \in ({V},1],\\
    & A_n \in (S_{d,2}, {V^{-\frac 1d}}S_{d,2}),
\end{align*}
for all $n \in\N$. By similar scaling arguments, for every $n\in\N$ we have
\[
1-\eta =  \int_{B^{Y_n}_1(y_n)} |u_{\sigma_n}|^{2^*}\, \d \mu_n\qquad \text{and}\qquad  \| u_{\sigma_n}\|_{L^{2^*}(\mu_n)}  \ge (A_n-1/n) \| \nabla u_{\sigma_n}\|_{L^2(\mu_n)}.
\]
We can then pass to a not relabelled subsequence so that $A_n\to A$, for some $A>0$ finite.

We are therefore in position to invoke Theorem \ref{thm:CC_Sobextremals} with $\eta \coloneqq \eta_N/2$ and get that,up to a subsequence, $(Y_n,\rho_n,\mu_n,y_n)$ pmGH-converges to some  pointed  ${\sf RCD}(0,d)$ space $(Y,\rho,\mu, z)$ satisfying
\begin{equation}
\|u\|_{L^{2^*}(\mu)} \le A \|\nabla u\|_{L^2(\mu)},\qquad \forall u \in W^{1,2}(Y).
\label{eq:ausixliary}
\end{equation}
Recall that the sharpness result in Theorem \ref{thm:sob avr rcd} gives automatically that ${\sf AVR}(Y)^{-\frac 1d}S_{d,2} \le A$ (see also Remark \ref{rem:sharp}). Moreover, we also have by Theorem \ref{thm:CC_Sobextremals} that $u_{\sigma_n}$ convergence $L^{2^*}$-strong  to a non-zero function $u_\infty \in W^{1,2}_{loc}(Y)$, that $\|\nabla u_{\sigma_n}\|_{L^2(\mm_{\sigma_n})}\to \|\nabla u_\infty\|_{L^2(\mu)}$ and that equality equality holds in \eqref{eq:ausixliary} with $u_\infty$. Thus, since $u_\infty$ is non-zero, we have the non-trivial estimate
\begin{align*}
     {\sf AVR}(Y)^{-\frac 1d}S_{d,2} \| \nabla u_\infty \|_{L^2(\mu)} &\ge  \|u_\infty \|_{L^{2^*}(\mu)} =\lim_{n\to \infty} \|u_{\sigma_n} \|_{L^{2^*}}(\mu_n) \\
     &\ge  \lim_{n\to \infty} (A_n -1/n)\|\nabla u_{\sigma_n}\|_{L^2(\mu_n)} = A\|\nabla u_\infty\|_{L^2(\mu)},
\end{align*}
giving that
\[ 
{\sf AVR}(Y)^{-\frac 1d}S_{d,2} =A,\qquad \text{hence also}\qquad {\sf AVR}(Y)=\lim_{n\to\infty}{\sf AVR}(M_n).
\]
 In particular, $u_\infty$ satisfies the rigidity part of Theorem \ref{thm:sob avr rcd}, which gives that $Y$ is a $N$-Euclidean cone and, for some tip  $z_0 \in Y$ and suitable $a \in \R, b>0$, we find
\[
u_\infty  = \frac{a}{(1+b\rho^2(\cdot,z_0))^{\frac{d-2}{2}}}.
\]

Take now any $z_n \in Y_n$ with $z_n \to z_0$, thanks to the underlying pmGH-convergence. Writing $f(t) \coloneqq a(1+bt^2)^{\frac{2-N}{2}}$ for every $t \in \R^+$, recalling $|f|^{2^*}, |f'|^2\le C t^{-2N+2}$, it is not difficult to show that  $f\circ \rho_n(\cdot,z_n)$ converges $L^{2^*}$-strong to $u_\infty$ and that $ |\nabla(f\circ \rho_n(\cdot,z_n))| $ converges $L^2$-strong to $|\nabla u_\infty|$ (see \cite[Lemma 7.2]{NobiliViolo24}, using $|\nabla (f\circ\rho_n(\cdot,z_n))| = |f'|\circ \rho_n(\cdot,z_n) $ by the chain rule). Using now that $\|\nabla u_{\sigma_n}\|_{L^2(\mm_{\sigma_n})}\to \|\nabla u_\infty\|_{L^2(\mu)}$ and $\| \nabla (f\circ \rho_n(\cdot,z_n))\|_{L^2(\mu_n)} \to \|\nabla u_\infty\|_{L^2(\mu)}$, and by polarization identity using also \cite[Lemma A.5]{NobiliViolo24}, we then deduce
\begin{equation}
   \lim_{n\to\infty} \| \nabla \big( u_{\sigma_n} - f\circ \sfd_{\sigma_n}(\cdot, z_n)\big)\|_{L^{2}(\mm_{\sigma_n})} = 0.\label{eq:un AVR gradient recovery}
\end{equation}
Scaling back, \eqref{eq:un AVR gradient recovery} becomes
\[
   \lim_{n\to\infty} \| \nabla \big( u_n - (\sigma_n^{N/2^*}f)\circ (\sigma_n\sfd_n(\cdot, z_n) ) \big)\|_{L^{2}(\mm_{\sigma_n})} = 0.
\]
This means that the sequence  $v_n \coloneqq   a\sigma_n^{d/2^*} (1+  b\sigma_n^2\sfd_n(\cdot,z_n)^2)^{\frac{2-d}{2}}$, satisfies
\[
   \limsup_{n\to\infty} \frac{\| \nabla ( u_n - v_n )\|_{L^2(M_n)}}{\|\nabla u_n\|_{L^2(M_n)}} =0,
\] 
having also used that $\|\nabla u_n\|_{L^2(M_n)}\ge S_{d,2}^{-1}\avr_{g_n}(M_n)^{1/d} \|u_n\|_{L^{2^*}(M_n)}\ge S_{d,2}^{-1}v^{\frac 1d}.$ Since $v_n = u_{a_n,b_n,z_n}$ for $a_n = a\sigma_n^{d/2^*},b_n = b\sigma_n^2$, we find a contradiction with \eqref{eq:contradiction AVR}. The proof is therefore concluded.
\end{proof}
\subsection{Comments for the compact case}
In this part, we briefly comment on the necessary modification of the proof of Theorem \ref{thm:qualitative SobAVR intro} to obtain the result iii) in Theorem \ref{thm:into main compact}. Even though the main proof strategy is the same, there are some interesting difficulties to be faced.

\medskip

First, iii) in Theorem \ref{thm:into main compact} will be obtained again by a contradiction argument, i.e.\ we suppose there exist $\eps>0,$ a sequence $d$-dimensional Riemannian manifolds $(M_n,g_n)$ with ${\sf Ric}_{g_n}\ge d-1$ and non-constant functions $u_n\colon M_n \to \R$, for simplicitly so that $\|u_n\|_{L^{2^*}(\nu_n)}=1$, for all $n\in\N$ so that 
\[
\frac{\|u_n\|_{L^{2^*}(\nu_n)}^2-\|u_n\|_{L^2(\nu_n)}}{\| \nabla u_n\|_{L^2(\nu_n)}} > \frac{2^*-2}{d}-\delta_n
\]
for some  $\delta_n\downarrow 0$, but so that for any $n\in\N$, it holds
\begin{equation}
    \inf_{a,b,z} \| u_n- v_{a,b,z}\|_{L^{2^*}(\nu_n)} + \| \nabla( u_n- v_{a,b,z})\|_{L^2(\nu_n)} > \eps,
\label{eq:contradiction intro}
\end{equation}
where  $v_{a,b,z} = a(1 - b\cos(\sfd_{g_n}(\cdot,z))^{\frac{2-d}{2}}$. Here $\nu_n$ is the renormalized volume measure on $M_n$.

As in the proof of Theorem \ref{thm:qualitative SobAVR intro}, we choose points $y_n \in M_n$ and constants $\sigma_n>0$ so that, defining
\begin{equation}
(Y_n,\rho_n,\mu_n)\coloneqq (M_n,\sigma_n\sfd_{g_n},{\rm Vol}_{g_n}(M_n)^{-1}\sigma_n^d{\rm Vol}_n),\qquad 
 u_{\sigma_n} = \sigma_n^{-d/2^*}u_n,
\label{eq:transformation intro}
\end{equation}
we have
\[
\int_{B_1^{Y_n}(y_n)} |u_{\sigma_n}|^{2^*} \, \d \mu_n=1-\eta,
\]
for $\eta>0$ suitable. The spaces $(Y_n,\rho_n,\mu_n)$ are in particular, by compatibility properties, normalized ${\sf RCD}(0,d)$ spaces. Note that it can happen that $\sigma_n \uparrow \infty$ (corresponding to a concentrating behavior of the sequence $u_n$) while instead $\sigma_n \ge c>0$ uniformly, by the maximal diameter theorem. 

Thanks to Gromov's precompactness theorem \cite{Gromov07} it is possible again to show that, up to a subsequence, $(Y_n,\rho_n,\mu_n,y_n)$ converges in the pmGH-sense to a limit pointed $\RCD(0,d)$ space $(Y,\rho,\mu, y)$. As already remarked, $Y$ might be also unbounded.

Using a generalized version of the concentration compactness result in Theorem \ref{thm:CC_Sobextremals} in this setting, see \cite[Theorem 6.2]{NobiliViolo24}), it is again possible to show that, up to a further subsequence, $u_{\sigma_n}$ converges $L^{2^*}$-strongly to some $0\neq u_\infty \in L^{2^*}(\mu)$. It also follows that $u_\infty$ is extremal for a formal limit Sobolev inequality on $Y$, that reads as
\[
\| u_\infty \|_{L^{2^*}(\mu)}^2 \le \frac{2^*-2}{d}\| \nabla u_\infty\|_{L^2(\mu)}^2 + \sigma^{-2 }\|u_\infty\|_{L^2(\mu)},\qquad \forall u \in W^{1,2}(Y),
\]
where $\sigma_n \to \sigma \in (0,\infty]$, and the $L^2$-term is present if $\sigma<\infty$, while it is not if $\sigma=\infty$. Recall that the case $\sigma=0$ is not possible by the maximal diameter theorem. The key point is the following dichotomy scenario:
\begin{align*}
    &\text{Concentration}  &\Rightarrow \qquad &Y \text{ is a $d$-Euclidean cone and } u \text{ is a Euclidean bubble} \\
    &\text{Non-concentration} &\Rightarrow \qquad &Y \text{ is a $d$-spherical suspension and } u \text{ is a spherical bubble}
\end{align*}
The first is settled again by the rigidity of Sobolev inequalities in non-compact setting, c.f. Theorem \ref{thm:sob avr rcd}. The second case requires a similar analysis for compact settings, see \cite[Theorem 5.2]{NobiliViolo24}. The proof will be then completed by carefully bringing back this information from $u_\infty$ to the sequence $u_n$ to find a contradiction with \eqref{eq:contradiction intro}. 
\section{Further results and open questions}\label{sec:further}
\subsection{Stability in the AB-program}\label{sec:AB-program}
Here we report some related results obtained in \cite{NobiliParise24}. We first introduce briefly the AB-program in geometric analysis following \cite{Hebey99}. We refer to this book for a detailed introduction and a complete list of references.

\medskip

Given $d>2$ and $(M,g)$ a closed $d$-dimensional smooth Riemannian manifolds, we can consider for constants $A,B>0$ the following type of Sobolev inequality
\begin{equation}
\|u\|_{L^{2^*}(M)}^2 \le A\|\nabla u\|_{L^2(M)}^2 + B \| u\|_{L^2(M)}^2,\qquad \forall u \in W^{1,2}(M).
\label{eq:SobAB}
\end{equation}
In the AB-program we consider the following Sobolev constants:
\[
\begin{array}{cc}
\alpha(M):= \inf \{ A \ : \  \eqref{eq:SobAB} \text{ holds for some $B$}  \}, &\quad  \beta(M) := \inf \{ B \ : \  \eqref{eq:SobAB} \text{ holds for some $A$} \},
\end{array}
\]
where we set the infimum over an empty set equal to $+\infty$. It is rather easy to see (e.g.\ \cite[Sec 4.1]{Hebey99}) that the latter satisfies
\[
\beta(M) = {\rm Vol}_g(M)^{-2/d}.
\]
On the other hand, the value of $\alpha(M)$ is linked to the sharp Euclidean Sobolev constant. More precisely, in \cite{Aubin76-2} (see also \cite[Theorem 4.5]{Hebey99}) we have
\begin{equation}\label{eq:alfa manifold}
	\alpha(M)= S_{d,2}^2.
\end{equation}
A more subtle question is whether these two constants are actually attained, i.e. if they are minima. For instance, for $M=\S^d$ we have in \eqref{eq:Sob sphere} the validity of a Sobolev inequality with $A=\alpha(\S^d)$ and $B = \beta(\S^d)$ with extremizers given by \eqref{eq:spherical bubbles}. 

More generally (see \cite{Bakry94},\cite[Theorem 4.2]{Hebey99} and \cite{HebeyVaugon96}) it is known that these two constants are always attained, and thus it makes sense to fix $B = \beta(M)$ and proceed with the $A$-part of this program, i.e.\ minimizing over the admissible $A>0$ and consider more generally all sub-critical ranges 
\begin{equation}
\|u\|_{L^{q}(M)}^2 \le A\|\nabla u\|_{L^2(M)}^2 + {\rm Vol}_g(M)^{2/q-1} \| u\|_{L^2(M)}^2,
\label{eq:Sob_qA}
\end{equation}
for $q \in (2,2^*]$. Then, we can define the corresponding notions of second-best optimal constant
\[
A_q^{\rm opt}(M) := \inf \{A>0 \colon  \eqref{eq:Sob_qA}\text{ holds with }A\} \cdot {\rm Vol}_g(M)^{1-2/q}.
\]
Differently from $\alpha(M),\beta(M)$ this constant is automatically a minimum. Moreover, universal bounds depending on the geometry of $M$ are given for $ A_q^{\rm opt}(M)$  in \cite[Theorem 4.4]{DH02} (see also \cite[Proposition 5.1]{NobiliViolo21} for general $q$).
\begin{remark}
    \rm
    The renormalization appearing in the above is non-standard in the AB-program literature.  Thanks to it, the above notation is consistent (modulo adding the subscript $q$) with that introduced in \eqref{intro: AM < AS} under Ricci curvature lower bounds. In fact, denoting as customary by-now in this note, the renormalized volume form $\nu= {\rm Vol}_g/{\rm Vol}_g(M)$, the optimal Sobolev inequality \eqref{eq:Sob_qA} equivalently reads
    \begin{equation}
        \|u\|_{L^{q}(\nu)}^2 \le A_q^{\rm opt}(M)  \|\nabla u\|_{L^2(\nu)}^2 +  \| u\|_{L^2(\nu)}^2.\label{eq:Sob_qA normalized}
    \end{equation}
    From here, we shall stick to this choice to be consistent in this note.  \fr
\end{remark}

Having established two classes of optimal Sobolev inequalities, it is natural to investigate the existence of extremal functions. For each $q \in (2,2^*]$, let us then define
\begin{align*}
    \mathcal{M}_q(A) &:= \{u \in W^{1,2}\setminus\{0\}  \colon  \text{equality holds with } u \text{ in }\eqref{eq:Sob_qA normalized} \}.
\end{align*}
Notice that $\mathcal M_q(A)$ always contains constant functions, hence it is never empty. However, we do not restrict our analysis to \emph{non-constant} extremal functions as these could be the only ones (see also \cite{Frank21}). Finally, we refer the reader to  \cite{Hebey01,DjadliDruet01,BarbosaMontenegro07,BarbosaMontenegro09,Barbosa10,BarbosaMontenegro12}, for more on extremal functions in the AB-program.
\medskip

Next, we present our main results achieved in \cite{NobiliParise24} and then comment briefly on related literature. Recall that differently from \cite{NobiliParise24}, here we are using a different renormalization for $A^{\rm opt}_{2^*}(M)$. By definition of $\alpha(M)$, we always have $A^{\rm opt}_{2^*}(M) {\rm Vol}_g(M)^{-2/d} \ge S^2_{d,2}$. However, assuming this inequality to be strict, we gain pre-compactness of (normalized) extremizing sequences via concentration-compactness methods, see \cite[Proposition 2.1]{NobiliParise24}. Under this assumption, or more easily in the sub-critical range, we can prove the following. 
\begin{theorem}\label{thm: stability AB program}
    Let $(M,g)$ be a closed $d$-dimensional Riemannian manifold, $d>2$. Assume that either $2<q<2^*$, or $q=2^*$ and $A^{\rm opt}_{2^*}(M) {\rm Vol}_g(M)^{-2/d} > S^2_{d,2}$. Then, there are constants $C>0,\gamma\ge 0$ depending on $(M,g)$ and on $q$ so that, for every $u \in W^{1,2}(M)\setminus\{0\}$, it holds
    \begin{equation}
        \frac{ A^{\rm opt}_q(M)\|\nabla u\|^2_{L^2(\nu)} +\|u\|_{L^2(\nu)}^2}{\| u\|_{L^{q}(\nu)}^2} - 1 \ge C \left( \inf_{v \in {\mathcal M}_q(A)} \frac{\| u- v\|_{W^{1,2}(M)} }{\|u\|_{W^{1,2}(M)}} \right) ^{2+\gamma}. \label{eq:deficit AoptB}
    \end{equation}
\end{theorem}
The above stability estimaate is \emph{strong}, as the full $W^{1,2}$-distance from a set a non-empty set of extremal functions is controlled by the related Sobolev deficit. The proof strategy is based on a recent argument that is particularly effective to treat general Riemannian settings where one cannot rely on underlying symmetries (as opposed to for instance, in \eqref{eq:spherical bubbles} for $\S^d$). This was pioneered in \cite{ChodoshEngelsteinSpolaorSpolaor23} for the Riemannian isoperimetric inequality. Later, and similar to our setting, this strategy was also used in \cite{EngelsteinNeumayerSpolaor22} to deduce quantitative stability properties of minimizing Yamabe metrics { on closed manifolds that are assumed non-conformally equivalent to the round sphere. In this case, the existence and pre-compactness properties of minimizing Yamabe metrics are known by the resolution of the Yamabe problem \cite{Trudinger68,Aubin76-3,Schoen84}.} The main idea of \cite{ChodoshEngelsteinSpolaorSpolaor23,EngelsteinNeumayerSpolaor22} that we also use to prove Theorem \ref{thm: stability AB program} is to combine a Lyapunov-Schmidt reduction argument \cite{Simon83} with the \L ojasiewicz inequality \cite{Lojasiewicz65} to produce local quantitative stability estimates. Then, we prove compactness results for almost minimizers (see \cite[Proposition 2.1]{NobiliParise24}) to combine a finite number of local estimates with uniform parameters via a local-to-global step. { The latter step is analogous to that of Bianchi-Egnell's work \cite{BianchiEgnell91} and consists in showing that the Sobolev deficit in the left of \eqref{eq:deficit AoptB} is far away from zero if $u$ varies at uniform positive $W^{1,2}$-distance from $\mathcal M_q(A)$}. The assumption $A^{\rm opt}_{2^*}(M) {\rm Vol}_g(M)^{-2/d} > S^2_{d,2}$ is crucial in the critical exponent  case.

\medskip

Given the above results, it is also natural to ask if our results hold \emph{sharp} with quadratic exponent, i.e.\ with $\gamma=0$. This is the typical desirable feature in stability problems such as for \eqref{eq:strong stability Eucl}. However, in \cite{ChodoshEngelsteinSpolaorSpolaor23,EngelsteinNeumayerSpolaor22}, the authors show that quadratic stability might fail, on specific manifolds. This degenerate phenomenon has been later deeply analyzed by \cite{Frank21}, and subsequently in \cite{FrankPeteranderl24} and \cite{BrigatiDolbeaultSimonov24,BrigatiDolbeaultSimonov24_2,BrigatiDolbeaultSimonov24_3,AndradekonigRatzkinWei24} in different contexts. Next, we present degenerate phenomena for Sobolev inequalities.
\begin{theorem}\label{thm:degenerate Aopt main}
    Let $q \in (2,2^*]$ and let $(M,g)$ be satisfying the hypothesis of Theorem \ref{thm: stability AB program}. Assume further that there are no non-constant extremal functions, i.e.\ $\mathcal M_q(A) = \{ c \colon c\in \R\setminus\{0\}\}.$ Then, the stability in Theorem \ref{thm: stability AB program} is degenerate, i.e.\ it must hold with some $\gamma >0$. 
\end{theorem}
We point out that this result is in line with the degenerate examples studied by Frank in \cite{Frank21} for the critical Sobolev inequality on $M=\mathbb S^1\left(\tfrac{1}{\sqrt{d-2}}\right)\times \mathbb S^{d-1}$ and, for sub-critical ones, on $M=\S^d$ (see \cite[Corollary 1.4]{NobiliParise24}). 
\subsection{Geometric stability: more functional inequalities}\label{sec:further results compact}
To prove ii) of Theorem \ref{thm:into main compact}, the key ingredient is a quantitative diameter-improvement of the P\'olya-Szeg\H o inequality. { We give here a short overview of this tool of independent interest. We shall stick for simplicity to the settings of smooth manifolds, even though the analysis conducted in \cite{NobiliViolo24_PZ} covers metric space settings.

\medskip 
Given $p\in(1,\infty)$ and a closed $d$-dimensional Riemannian manifold $(M,g)$ satisfying ${\sf Ric}_g\ge d-1$, we have the validity of the following P\'olya-Szeg\H{o} inequality
\begin{equation}
    \int|\nabla u|^p\,\d \nu \ge {\sf BBG}_d(\diam(M))^{p} \int_0^\pi |(u^*)'|^p\,\d \mm_{d-1,d},\qquad \forall u \in W^{1,p}(M), \label{eq:polya BBG}
\end{equation}
where $\nu = {\rm Vol}_g/{\rm Vol}_g(M)$ is the renormalized volume form, $\mm_{d-1,d}$ is the weight defined by
\[
\mm_{d-1,d} \coloneqq \frac{\sin^{d-1}(t)\,\d t\mres{(0,\pi)} }{\int_0^{\pi} \sin^{d-1}(t)\,\d t},
\]
and $u^*$ is the decreasing rearrangement of $u$ into the model interval $(0,\pi)$ with respect to $\mm_{d-1,d}$
\[
    (0,\pi)\ni x \mapsto u^*(x) \coloneqq \inf\{ t >\essinf u \, \colon \, \nu(\{u>t\})< \mm_{d-1,d}((0,x)) \}.
\]
By precomposing $u^*$ with the geodesic distance on $\S^d$ from e.g.\ the south pole, it is possible to relate this notion of decreasing rearrangement with that of \emph{spherical rearrangement} previously considered in the literature (see, e.g., \cite{BerardMeier1982}). The function $ (0,\pi] \ni D\mapsto {\sf BBG}_d(D) \in [1,+\infty)$ is explicit and called after the names of the authors in \cite{BerardBessonGallot85} where it was studied. 

The inequality \eqref{eq:polya BBG} can be proved by classical rearrangement arguments relying on an generalization of the celebrated L\'evy-Gromov inequality \cite{Gromov07}. In this generalization obtained in \cite{BerardBessonGallot85} the dependence on the diameter appears explicitly through the constant ${\sf BBG}_d(\diam(M))$ (see also \cite{Milman15,CavallettiMondinoSemola23} for related analysis) so that the corresponding rearrangement inequality would take the form \eqref{eq:polya BBG}. We refer to \cite[Theorem 1.2]{NobiliViolo24_PZ} for the proof of \eqref{eq:polya BBG} working also in nonsmooth settings (see also the previous work \cite{MondinoSemola20}). Moreover, since $\diam(M)\le \pi$ under the current assumptions, the key point (see \cite[Lemma 3.3]{CavallettiMondinoSemola23}) is the validity of
\[
    {\sf BBG}_d(D)^2 -1 \gtrsim (\pi-D)^d,\qquad\text{as }D\uparrow\pi.
\]
By combination of this asymptotic with \eqref{eq:polya BBG}, we can derive further} applications for different functional inequalities under positive Ricci lower bounds as we shall explore next.

\medskip

Recall that, if $(M,g)$ is so that ${\sf Ric}_g\ge d-1$, then comparison estimates for spectral quantities and sharp functional constants are in place. Here, we shall deal with Lichnerowicz spectral gap inequalities and log-Sobolev constants, as well as sub-critical optimal Sobolev inequalities. We refer to the book \cite{BakryGentilLedoux14} for a complete treatment and to \cite{CavallettiMondino17} for analogous results in non-smooth settings. For $d \in N,p\in(1,d)$, we denote by $\lambda_p(\S^d)$ the first non-trivial (Neumann) $p$-eigenvalue, also called $p$-spectral gap, of the $d$-dimensional round sphere. The next result is reported from \cite[Theorem 1.6]{NobiliViolo24_PZ} only in the context of smooth manifolds for simplicity, but works in more general and possibly weighted settings of non-smooth spaces.
\begin{theorem}\label{thm:main quant compact spaces}
For all $d\in (1,\infty)$, $p\in(1,\infty)$ and $q\in (2,2^*]$ (if $d>2$) there are constants $C_d>0$, $C_{d,p}>0$ and $C_{d,q}>0$ such that the following holds.
Let $(M,g)$ be a $d$-dimensional Riemannian manifolds with ${\sf Ric}_g\ge d-1$. Then:
\begin{itemize}
    \item[{\rm i)}] for all $u \in W^{1,p}(M)$ non-zero with $\int u|u|^{p-2}\,\d{\rm Vol}_g =0$ it holds
        \[
        \big(\pi- \diam(M)\big)^{d} \le C_{d,p} \left( \left(\int|u|^p\,\d {\rm Vol}_g\right)^{-1}\int |\nabla u|^p\,\d {\rm Vol}_g  - \lambda_p(\S^d)\right);
        \]
        \item[{\rm ii)}] it holds
        \[
        \big(\pi- \diam(M)\big)^d \le C_{d,q} \left(  \frac{q-2}{d} - A^{\rm opt}_q(M)\right);
        \]
        
        \item[{\rm iii)}] for  all $u \in W^{1,2}(M)$ non-constant with    $\int |u|\,\d{\rm Vol}_g=1$ it holds
        \[
        \big(\pi- \diam(M)\big)^d \le C_d \left(  \Big(\int |u|\log|u|\,\d{\rm Vol}_g\Big)^{-1}\int_{\{|u| >0\}}\frac{|\nabla u|^2}{|u|}\,\d{\rm Vol}_g- 2d \right).
        \]
\end{itemize}
\end{theorem}
The result i) for $p=2$ was already deduced in \cite{CavallettiMondinoSemola23}, where also a deep quantitative analysis has been carried out for almost eigenfunctions. We mention that a similar analysis was previously performed in \cite{CavallettiMaggiMondino19} to study quantitative stability results for the L\'evy-Gromov isoperimetric inequality.

The result ii) generalizes to all the sub-critical range the first two conclusions of Theorem \ref{thm:into main compact}. Indeed, we recall from \cite{Said83,BakryGentilLedoux14,Ledoux00} that an analogous comparison estimate as for \eqref{intro: AM < AS} holds in this range
\[
A^{\rm opt}_q(M)\le \frac{q-2}{d} =A^{\rm opt}_q(\S^d),\qquad\forall q \in (2,2^*],
\]
where, the latter equality was studied in \cite{Beckner93} for the sub-critical range. 

Finally, result iii) instead is a quantitative improvement of the sharp dimensional log-Sobolev inequality in this setting, since on the round sphere the optimal constant is precisely $2d$, see \cite{BakryGentilLedoux14}.

\subsection{Open questions}
We conclude here by raising some natural open questions related to the theory and the statements presented in this overview.

\medskip 

Given that we presented qualitative functional stability results for Sobolev inequalities under Ricci lower bounds, we raise here two natural open questions:
\begin{question}\label{Q1}\em 
    Is it possible to perform a quantitative analysis for Theorem \ref{thm:qualitative SobAVR intro}?
\end{question}
\begin{question}\label{Q2}\em 
     Is it possible to perform a quantitative analysis for iii) in Theorem \ref{thm:into main compact}?
\end{question}
Clearly, since the main technique used is a (concentration) compactness argument, to address these questions one needs to change the approach completely (or, at least, complement it with a quantitative local analysis around bubbles). Moreover, rearrangement techniques \'a la P\'olya-Szeg\H{o} seem to give no useful quantitative information at the level of almost extremal functions (even though, they are powerful enough to deal with quantitative geometric stability results, c.f.\ Theorem \ref{thm:main quant compact spaces}). Finally, a classical strategy \'a la Bianchi-Egnell \cite{BianchiEgnell91} is not expected to work for both the above questions since, in general, Euclidean or spherical bubbles are not extremal functions on $(M,g)$, unless $M$ is isometric to the corresponding model space with constant curvature (where both questions are already settled).

A rather technical, yet promising, approach would be to argue via $L^1$-localization methods, see \cite{CavallettiMondino17-Inv}. In fact, this has already proved successful in understanding the shape of almost extremizers in different variational problems, see \cite{CavallettiMaggiMondino19} for the L\'evy-Gromov isoperimetric inequality and \cite{CavallettiMondinoSemola23,FathiGentilSerres22} for spectral gap inequalities. However, we expect Question \ref{Q1} to be considerably harder compared to Question \ref{Q2} due to the non-compact nature of the underlying manifold geometry and due to presence of the ${\sf AVR}$-parameter (recall that $\delta$ in Theorem \ref{thm:qualitative SobAVR intro} depends on $V$). 

\medskip

We conclude this note by raising an open question related to Theorem \ref{thm: stability AB program}. In the critical exponent case $q=2^*$, a crucial assumption was made. Namely (and with the same notation as in there), it is asked that $A^{\rm opt}_{2^*}(M){\rm Vol}_g(M)^{-\frac 2d}>S^2_{d,2}$, where we stress that the inequality is required to be \emph{strict}. Since, in general, the greater or equal inequality is always in place (recall  $\alpha(M) =S^2_{d,2}$ and the discussion in Section \ref{sec:AB-program}), it is natural to wonder if Theorem \ref{thm: stability AB program} is still valid when $A^{\rm opt}_{2^*}(M){\rm Vol}_g(M)^{-\frac 2d}=S^2_{d,2}$. In this case, if one looks at extremizing sequences optimizing the constant $A^{\rm opt}_{2^*}(M)$ in \eqref{eq:Sob_qA normalized}, then we know that concentrating phenomena might occur, see \cite[Proposition 2.2]{NobiliParise24}. In light of this, we can raise the following:
\begin{question}\em 
    Consider a closed $d$-dimensional Riemannian manifold $(M,g)$ possibly so that $A^{\rm opt}_{2^*}(M){\rm Vol}_g(M)^{-\frac 2d}=S^2_{d,2}$. Does the  conclusion of Theorem \ref{thm: stability AB program} with $q=2^*$ holds true, possibly replacing the set $\mathcal M_{2^*}(A)$ with $\widetilde{\mathcal M}_{2^*}(A)\coloneqq \mathcal M_{2^*}(A) \cup \left\{ a(1+b\sfd_g(\cdot,p)^2)^{\frac{2-d}{2}} \colon a\in\R, b>0, p \in M \right\}$ ?
\end{question}
Notice that set $\widetilde{\mathcal M}_{2^*}(A)$ is obtained by adding to the true extremal set $\mathcal M_{2^*}(A)$ the family of Euclidean bubbles. The reason is that, since we have no information around the shape of extremal functions $\mathcal M_{2^*}(A)$, it is likely true that highly concentrated profiles of Euclidean bubbles can detect the possible concentrating behavior of certain extremizing sequences. For instance, similar phenomena are well-known to occur in the study of elliptic partial differential equations of critical order \cite{DruetHebeyFrederic03_Notice,DruetHebeyFrederic04,Struwe84}, see also the recent quantitative analysis \cite{ChenKim24}.

\medskip 
\noindent\textbf{Acknowledgments}. The author is a member of INDAM-GNAMPA and acknowledges the MIUR Excellence Department Project awarded to the Department of Mathematics, University of Pisa, CUP I57G22000700001.

I warmly thank I. Y. Violo for numerous inspiring discussions we had around the topics of this note. I also thank M. Pozzetta for useful remarks and comments on a preliminary version of this manuscript, D. Parise and E. Pasqualetto for their interest, M. Freguglia for pointing out relevant references and the anonymous referee for the careful reading and precise comments that substantially helped improving this manuscript.

Lastly, I would like to thank F. Cavalletti, M. Erbar, J. Maas and K. T. Sturm for the organization of the ``School and Conference on Metric Measure Spaces, Ricci Curvature, and Optimal Transport'' (MeRiOT 2024) and for the opportunity to write this work.

\def\cprime{$'$} \def\cprime{$'$}

\end{document}